\documentclass[11pt]{amsart}

\usepackage{amsthm}
\usepackage{amsmath}
\usepackage{amssymb}
\usepackage{amsrefs}
\usepackage{hyperref}
\usepackage{enumitem}

\usepackage[mathlines,modulo]{lineno}

\usepackage[color=yellow!20,linecolor=orange]{todonotes}

\theoremstyle{plain}
\newtheorem{theorem}{Theorem}[section]
\newtheorem{proposition}[theorem]{Proposition}
\newtheorem{lemma}[theorem]{Lemma}
\newtheorem{corollary}[theorem]{Corollary}
\newtheorem{conjecture}[theorem]{Conjecture}

\theoremstyle{definition}
\newtheorem{definition}[theorem]{Definition}
\newtheorem{claim}{Claim}[theorem]

\newcommand{\dash}{\nobreakdash-\hspace{0pt}}
\newcommand{\ba}{\backslash}
\newcommand{\cl}{\operatorname{cl}}
\newcommand{\st}{\operatorname{star}}
\newcommand{\mso}{\ensuremath{\mathit{MS}_{0}}}

\newcommand{\formula}[1]{\text{\fontfamily{cmss}\selectfont \textup{#1}}}
\newcommand{\ind}{\formula{Ind}}
\newcommand{\emp}{\formula{Empty}}
\newcommand{\sing}{\formula{Sing}}

\title[Defining bicircular matroids]{Defining bicircular matroids in monadic logic}

\author[Funk]{Daryl Funk}
\author[Mayhew]{Dillon Mayhew}
\author[Newman]{Mike Newman}

\begin{document}

\begin{abstract}
We conjecture that the class of frame matroids can be characterised by a sentence in the monadic second-order logic of matroids, and we prove that there is such a characterisation for the class of bicircular matroids.
The proof does not depend on an excluded-minor characterisation.
\end{abstract}

\maketitle

\section{Introduction}

The ability of monadic second-order logic to express graph properties has been investigated for some time \cites{MR1451381, MR1130377, MR743800}.
This study continues to be motivated by Courcelle's famous theorem showing that such a property can be tested in polynomial time for any class of input graphs with bounded tree-width \cite{MR1042649}.
On the other hand, the literature on matroid properties that can be expressed in monadic second-order logic is more recent, and sparser \cites{MR2081597,MR3803151}.
Our aim here is to add to this literature, by making progress towards the following conjecture.

\begin{conjecture}
\label{conj_mainconjecture}
There is a sentence in the monadic second-order language of matroids which characterises the class of frame matroids.
\end{conjecture}

The \emph{monadic second-order logic for matroids} (\mso) is formally defined in Section \ref{logic}.
We also conjecture that there are \mso\dash characterisations of Zaslavsky's class of lift matroids \cite{MR1088626} and the class of quasi-graphic matroids \cites{MR4037634, MR3742182}.
Let $G$ be a graph.
If $X$ is a set of edges in $G$, then $V(X)$ is the set of vertices incident with at least one edge in $X$, and $G[X]$ is the subgraph with $X$ as its edge-set and $V(X)$ as its vertex-set.
A \emph{bicycle} of $G$ is a minimal set, $X$, of edges such that $G[X]$ is connected and contains more than one cycle.
If $X$ is a bicycle, then $G[X]$ is a \emph{theta subgraph}, a \emph{loose handcuff}, or a \emph{tight handcuff}.
A theta subgraph consists of two distinct vertices and three internally disjoint paths that join them.
A loose handcuff is a pair of vertex-disjoint cycles, along with a minimal path that joins them, and a tight handcuff is a pair of edge-disjoint cycles that have exactly one vertex in common.
A \emph{linear class of $G$} is a set of cycles, $\mathcal{B}$,  such that no theta subgraph of $G$ contains exactly two cycles in $\mathcal{B}$.
Given such a class, there is a corresponding \emph{frame matroid} $F(G,\mathcal{B})$.
The ground set of $F(G,\mathcal{B})$ is the edge-set of $G$, and
the circuits of $F(G,\mathcal{B})$ are exactly the edge-sets of cycles in $\mathcal{B}$, along with the bicycles that do not contain any cycles in $\mathcal{B}$.

We are quite optimistic that Conjecture \ref{conj_mainconjecture} can be positively resolved.
In this paper, we establish the monadic definability of an important subclass of frame matroids.
A matroid is \emph{bicircular} if it is equal to $F(G,\mathcal{B})$ for some graph $G$, where $\mathcal{B}$ contains no cycle with more than one edge.
We write $B(G)$ for $F(G,\emptyset)$.
Thus the circuits of $B(G)$ are precisely the bicycles of $G$.
Every bicircular matroid is the direct sum of a rank-zero uniform matroid with a matroid of the form $B(G)$, so every connected bicircular matroid is either equal to $B(G)$ for some $G$, or is $U_{0,1}$.
Our main result is as follows:

\begin{theorem}
\label{thm_main}
There is a sentence in the monadic second-order language of matroids which characterises the class of bicircular matroids.
\end{theorem}

Let $\mathcal{M}$ be a class of matroids.
When we say that $\mathcal{M}$ has a \emph{decidable monadic second-order theory}, we mean that there exists a Turing machine that will accept as input any monadic second-order sentence.
The machine then decides in finite time whether or not the sentence holds for every matroid in $\mathcal{M}$.
Thus a decidable class has a theorem-testing machine which runs in finite time (even though the class itself may be infinite).
Let $\mathcal{M}$ be a minor-closed class of $\mathbb{F}$\dash representable matroids (where $\mathbb{F}$ is a finite field).
Hlin\v{e}n\'{y} and Seese have proved that $\mathcal{M}$  has a decidable monadic theory if and only if it has bounded branch-width \cite{MR2202497}.
Combining Theorem \ref{thm_main} with results from \cite{FMN-I} and \cite{FMN-II} produces the corollary that a class of bicircular matroids with bounded branch-width has a decidable monadic second-order theory.

Any minor-closed class of matroids with finitely many excluded minors can be characterised by a monadic second-order sentence \cite{MR2081597}.
The class of bicircular matroids is minor-closed.
DeVos and Goddyn announced that this class is characterised by finitely many excluded minors, but this proof did not appear in the literature.
A new proof has now been produced by DeVos, Funk, and Goddyn \cite{DFG}.
From their result we can deduce our main theorem, so one can reasonably ask why we present an alternative proof of Theorem \ref{thm_main} in this article.
We have two reasons.
Firstly, the proof by DeVos, Funk, and Goddyn does not provide an explicit list of excluded minors.
This means that we cannot construct an explicit sentence that characterises the class via excluded minors.
In contrast, our proof of Theorem \ref{thm_main} is constructive.
One could, with time and effort, use the results contained in this paper to write an explicit (and long) sentence characterising bicircular matroids.
Our second reason is that we hope to use the experience we gain in proving Theorem \ref{thm_main} as a stepping stone towards proving Conjecture \ref{conj_mainconjecture}.
The class of frame matroids is minor-closed, but it has infinitely many excluded minors \cite{MR3856704}.
Therefore no monadic characterisation of that class can rely on a list of excluded minors.
In order to prove Theorem \ref{thm_main}, we need to establish new structural facts about bicircular matroids.
We think that similar ideas will be required to prove Conjecture \ref{conj_mainconjecture}.

We conclude this introduction with some conjectures about \emph{gain-graphic} matroids.
These are frame matroids where a linear class of cycles is determined by an assignment of group elements to the edges of a graph.
Gain-graphic matroids over finite groups are restrictions of \emph{Dowling geometries} \cite{MR307951}, and they play an important role in the qualitative structural characterisation of minor-closed classes of matroids representable over a finite field, due to Geelen, Gerards, and Whittle \cite{MR3184116}.
This set of authors has also shown that when $\mathbb{F}$ is a finite field, any minor-closed class of matroids representable over $\mathbb{F}$ has finitely many excluded minors \cite{MR3221124}.
We conjecture that analogous results hold for classes of matroids that are gain-graphic over finite groups.

\begin{conjecture}
\label{gaingraphicrota}
Let $H$ be a finite group.
Let $\mathcal{M}$ be any minor-closed class of $H$\dash gain-graphic matroids.
Then $\mathcal{M}$ has only finitely many excluded minors, and can hence be defined in monadic second-order logic.
\end{conjecture}

The following conjecture is widely believed, but may not have appeared in print.

\begin{conjecture}
\label{gaingraphicwqo}
Let $H$ be a finite group.
The class of $H$\dash gain graphic matroids is well-quasi-ordered.
\end{conjecture}

Conjecture \ref{gaingraphicwqo} implies any minor-closed class of $H$\dash gain-graphic matroids has at most finitely many excluded minors that are themselves $H$\dash gain-graphic.
It is expected that the project of Geelen, Gerards, and Whittle will resolve Conjecture \ref{gaingraphicwqo}, at least in the case that $H$ is abelian (see \cite{Huynh09}*{Theorem 2.5.2}).
Note that the finiteness of $H$ is required for Conjectures \ref{gaingraphicrota} and \ref{gaingraphicwqo}, because DeVos, Funk, and Pivotto have shown that when $H$ is infinite, the class of $H$\dash gain-graphic matroids has infinitely many excluded minors and is not well-quasi-ordered \cite{MR3267062}.

The results in \cites{MR3221124, MR3803151} provide a clear dichotomy for matroid representation over fields: the class of $\mathbb{F}$-representable matroids is definable in monadic second-order logic if and only if $\mathbb{F}$ is finite.
The next conjecture, along with Conjecture \ref{gaingraphicrota}, would imply that a similar dichotomy exists for classes of gain-graphic matroids.

\begin{conjecture}
\label{conj_infinitegroup}
Let $H$ be an infinite group.
The class of $H$-gain-graphic matroids cannot be characterised by a sentence in monadic second-order logic.
\end{conjecture}

In fact we conjecture that when $H$ is infinite, the class of $H$-gain-graphic matroids cannot be characterised in \emph{counting monadic second-order logic}, which augments standard monadic second-order logic with predicates that allow us to say when a set has cardinality $p$ modulo $q$.
It is not too difficult to see that the techniques in \cite{MR3803151} allow us to prove Conjecture \ref{conj_infinitegroup} when $H$ contains elements of arbitrarily high order.
Thus the conjecture is unresolved only for infinite groups with finite \emph{exponent} (the least common multiple of all the orders of elements).
As an example of such a group, we could take a direct product of infinitely many copies of a finite group.
A \emph{Tarski monster} group would give a more sophisticated example.

In Section \ref{prelims} we introduce essential matroid ideas.
Section \ref{logic} is dedicated to introducing \mso, and discussing its power to define matroidal concepts.
In Section \ref{3-connected-case} we prove Theorem \ref{thm_main} for $3$\dash connected matroids.
Finally, in Section \ref{Reduceto3}, we show that it is possible to make statements in \mso\ about the $3$\dash connected components of a matroid.
This leads to the conclusion that a characterisation of the $3$\dash connected matroids in a class gives a full characterisation of the class, as long as it is closed under $2$\dash sums (Corollary \ref{definedviacomponents}).
Unfortunately the class of bicircular matroids is not closed under $2$\dash sums, so we must do additional work to analyse the decomposition of bicircular matroids into $3$\dash connected components (Theorem \ref{treebeard}).

\section{Preliminaries}
\label{prelims}

We say \emph{set-system} to mean what is sometimes called a hypergraph: a pair $(E,\mathcal{I})$ where $E$ is a finite set and $\mathcal{I}$ is a collection of subsets of $E$.
We refer to $E$ as the \emph{ground set} of the set-system.
Let $M$ be a matroid.
We say that the flat $Z$ is \emph{cyclic} if there are no coloops in the restriction $M|Z$.
Let $e$ and $f$ be elements of $E(M)$.
Then $e$ and $f$ are \emph{clones} if every cyclic flat of $M$ that contains one of $e$ and $f$ contains both. 
A \emph{clonal class} is a maximal set, $C$, such that every pair of distinct elements in $C$ is a clonal pair.

Let $M$ be a matroid, and let $A$ be a subset of the ground set.
Let $B$ be the complement of $A$.
Then $\lambda(A)$ is defined to be $r(A)+r(B)-r(M)$, which is equal to $r(A)+r^{*}(A)-|A|$.
If $\lambda(A)<k$, then $A$ is \emph{$k$\dash separating}.
A subset of $E$ is a \emph{separator} if it is $1$\dash separating.
A singleton set $\{e\}$ is a separator if and only if $e$ is a loop or a coloop.
In this case, we say that $e$ itself is a separator.
A minimal non-empty separator is a \emph{component} of $M$.
A set $C$ is a component if and only if it is maximal with respect to every pair of distinct elements in $C$ being contained in a circuit of $M|C$.
When $C$ is a component, we blur the distinction between the set $C$ and the restriction $M|C$.
If $A$ is $k$\dash separating and $|A|,|B|\geq k$, then the partition $(A,B)$ is a \emph{$k$\dash separation}.
A matroid is \emph{$n$\dash connected} if it has no $k$\dash separation such that $k<n$.
A $2$\dash connected matroid is said to be \emph{connected}.
Thus a matroid fails to be connected if and only if it has a non-empty separator that is not equal to the entire ground set.
Let $C^{*}$ be a cocircuit of the connected matroid $M$.
We say $C^{*}$ is \emph{non-separating} if $M\backslash C^{*}$ is a connected matroid.

Let $M_{1}$ and $M_{2}$ be matroids on the ground sets $E_{1}$ and $E_{2}$, respectively.
Assume that $E_{1}\cap E_{2}=\{e\}$, where $e$ is not a separator in either matroid.
Then $M_{1}\oplus_{2} M_{2}$ denotes the \emph{$2$\dash sum} of $M_{1}$ and $M_{2}$ along the \emph{basepoint} $e$.
Its ground set is $(E_{1}\cup E_{2})-e$, and its circuits are the circuits of $M_{1}\ba e$ and $M_{2}\ba e$, along with all sets of the form $(C_{1}-e)\cup (C_{2}-e)$, where $C_{i}$ is a circuit of $M_{i}$ containing $e$, for $i=1,2$.
A connected matroid has a $2$\dash separation if and only if it can be expressed as the $2$\dash sum of two matroids, each with at least three elements \cite{MR2849819}*{Theorem 8.3.1}.
The $2$\dash sum behaves well with respect to duality: $(M_{1}\oplus_{2} M_{2})^{*} = M_{1}^{*}\oplus_{2} M_{2}^{*}$.

Let $M = B(G)$ be a connected bicircular matroid.
If $e$ is an edge of the graph $G$, then $B(G)\ba e = B(G\ba e)$.
The rank function of $M$ is given by $r(X) = |V(X)|-a(X)$, where $a(X)$ is the number of acyclic components of $G[X]$. 
Hence the cocircuits of $M$ are exactly the complements of spanning trees, along with the sets that can be expressed as a disjoint union $A\cup B$, where $B$ is a bond of $G$, and $A$ is the complement of a spanning tree in one of the two components of $G\ba B$.

Let $v$ be a vertex in a graph $G$.
The \emph{vertex star} at $v$, denoted $\st(v)$, is the set of edges incident with $v$.

\begin{proposition} 
\label{stars_are_cocircuits} 
Let $G$ be a $2$\dash connected graph, and let $M$ be $B(G)$. 
Then $\st(v)$ is a cocircuit of $M$ if and only if $G-v$ contains a cycle.
\end{proposition}

\begin{proof}
Note that $G$ must contain a cycle.
Since $G-v$ is connected, as long as it contains a cycle, the star at $v$ is a minimal set of edges whose deletion from $G$ increases the number of acyclic components by one, namely by creating the acyclic component consisting of the isolated vertex $v$.
Therefore $\st(v)$ is a cocircuit.

For the converse, assume that $G-v$ is acyclic.
We let $e$ be a non-loop edge incident with $v$.
Then deleting $\st(v)-e$ from $G$ produces a graph that is acyclic, and therefore has exactly one more acyclic component than $G$.
Thus $\st(v)$ properly contains a cocircuit, and is itself not a cocircuit.
\end{proof}

In a graph, a \emph{leaf} is a degree-one vertex, and a \emph{pendent edge} is an edge that is incident with a leaf.
The next result follows from \cite{MR1120878}.

\begin{proposition}
\label{conn} 
Let $G$ be a connected graph with at least three vertices. 
\begin{enumerate}[label=\textup{(\roman*)}]
    \item $B(G)$ is connected if and only if $G$ is not a cycle and has no pendent edge. 
    \item $B(G)$ is $3$\dash connected if and only if $G$ has minimum degree at least three, no cut vertex, and no vertex incident with more than one loop.
\end{enumerate}
\end{proposition}

\begin{proposition}
\label{prop_onecomp}
Let $G$ be a connected graph, and let $M$ be $B(G)$.
Let $C^{*}$ be a cocircuit of $M$.
Then $M\ba C^{*}$ has at most one component with size greater than one.
\end{proposition}

\begin{proof}
Express $C^{*}$ as the disjoint union $A\cup B$, where $B$ is a bond of $G$.
Let $G_{1}$ and $G_{2}$ be the two components of $G\ba B$, and assume that $A$ is the complement $E(G_{2})-E(T_{2})$, where $T_{2}$ is a spanning tree of $G_{2}$.
Therefore every edge in $E(G_{2})-A$ is a coloop of $M\ba C^{*}$.
Now we see that any connected component of $M\ba C^{*}$ that contains more than one element must be contained in $E(G_{1})$.
Assume that $N_{1}$ and $N_{2}$ are two distinct components of $M\ba C^{*}$ and that $|N_{1}|,|N_{2}|>1$.
For $i=1,2$, let $C_{i}$ be a circuit of $M\ba C^{*}$ contained in $N_{1}$.
Then $G_{1}[C_{i}]$ is a connected subgraph of $G_{1}$ that contains at least two cycles.
As $G_{1}$ is connected, we can let $P$ be a minimal path that joins a vertex in $G_{1}[C_{1}]$ to a vertex in $G_{1}[C_{2}]$.
Now it is easy to see that there is a circuit of $M\ba C^{*}$ containing the edges of $P$, as well as cycles from both $G_{1}[C_{1}]$ and $G_{1}[C_{2}]$.
But this circuit is not contained in a connected component of $M\ba C^{*}$, and we have a contradiction.
\end{proof}

It is easy to see that if $e$ and $f$ are parallel edges in a graph $G$, then $e$ and $f$ are clones in the bicircular matroid $B(G)$.
The converse is almost always true, as the next result shows.
It follows with a small amount of analysis from \cite{MR3442541}*{Theorem 2.5}.

\begin{proposition}
\label{prop_bicircularclones}
Let $G$ be a graph with at least five vertices such that $M=B(G)$ is $3$\dash connected.
Assume that $G$ has no vertex whose deletion leaves at most one cycle.
If $e$ and $f$ are clones in $M$, then they are parallel edges in $G$.
\end{proposition}

\begin{proposition}
\label{stars} 
Let $G$ be a connected graph such that $M=B(G)$ is connected.
Any non-separating cocircuit of $M$ is a vertex star of $G$.
\end{proposition}

\begin{proof}
Then since $M$ is connected, so is $G$.
The case where $G$ has a single vertex is trivial, so we assume that it has more than one.
Let $C^{*}$ be a cocircuit of $M$, so that $C^{*} = A\cup B$, where $B$ is a bond of $G$ and $A$ is the complement of a spanning tree, $T$, in one of the two components in $G\backslash B$.
Assume that $C^{*}$ is non-separating and let $Z=E(M)-C^{*}$ be the complementary hyperplane.
Since $M|Z$ is connected, $G[Z]$ is connected.
This means $T$ contains no edges, and is thus a single vertex, $v$.
Thus $B$ is the bond consisting of non-loop edges incident with $v$, and $A$ is the set of loops incident with $v$.
Hence $C^{*} = \st(v)$.
\end{proof}

\section{Monadic second-order logic}
\label{logic}

In this section we describe monadic second-order logic, as used in \cite{MR3803151}.
We allow ourselves a countably infinite supply of variables, where each variable is to be interpreted as a subset of a ground set.
In addition, we use the predicate symbols $\ind$ and $\subseteq$, the logical connectives $\neg$ and $\land$, and the existential quantifier $\exists$.

We recursively describe the \emph{formulas} of monadic second-order logic.
Any variable that appears in a formula is either \emph{free} or \emph{bound}.
The following are \emph{atomic formulas}: the unary formula $\ind[X]$, which has the single (free) variable $X$, and the binary formula $X\subseteq X'$, which has $X$ and $X'$ as its free variables.
Now we recursively define non-atomic formulas.
If $\varphi$ is a formula, then so is $\neg \varphi$, and these two formulas have the same free variables.
If $X$ is a free variable in the formula $\varphi$, then $\exists X\varphi$ is a formula, and its free variables are exactly those of $\varphi$, except for $X$ (which is a bound variable of $\exists X\varphi$).
Finally, if $\varphi_{1}$ and $\varphi_{2}$ are formulas, and no variable is free in one of $\varphi_{1}$ and $\varphi_{2}$ while being bound in the other, then $\varphi_{1}\land \varphi_{2}$ is a formula.
Since we can always rename bound variables, the constraint on variables imposes no difficulty.
The free variables of $\varphi_{1}\land \varphi_{2}$ are exactly those variables that are free in either $\varphi_{1}$ or $\varphi_{2}$.
This completes the description of \mso\ formulas.
We will use parentheses freely to clarify the construction of formulas.
A formula with no free variables is a \emph{sentence}.

We use the notation $\varphi[X_{i_{1}},\ldots, X_{i_{n}}]$ to indicate that $\varphi$ is an \mso\ formula and that $\{X_{i_{1}},\ldots, X_{i_{n}}\}$ is its set of free variables.
In this case, $\varphi[X_{j_{1}}, \ldots, X_{j_{n}}]$ is the formula obtained from $\varphi[X_{i_{1}},\ldots, X_{i_{n}}]$ by replacing each appearance of $X_{i_{s}}$ in $\varphi$ with $X_{j_{s}}$, for $s=1,\ldots, n$.

Let $\varphi$ be an \mso\ formula and let $M=(E,\mathcal{I})$ be a set-system.
An \emph{interpretation} is a function, $\theta$, from the set of free variables in $\varphi$ to the power set of $E$.
We treat every function as a set of ordered pairs.
We will blur the distinction between a variable and its image under an interpretation when it is convenient to do so.

We are going to recursively define what it means for $(M,\theta)$ to \emph{satisfy} $\varphi$.
If $\varphi$ is the atomic formula $\ind[X]$, then $(M,\theta)$ satisfies $\varphi$ if and only if $\theta(X)\in\mathcal{I}$.
If $\varphi$ is the atomic formula $X\subseteq X'$, then $(M,\theta)$ satisfies $\varphi$ if and only if $\theta(X)\subseteq \theta(X')$.
Now we move to formulas that are not atomic.
Assume that $\varphi=\neg\psi$, for some formula $\psi$.
Then $(M,\theta)$ satisfies $\varphi$ if and only if $(M,\theta)$ does not satisfy $\psi$.
Next let $\varphi$ be $\varphi_{1}\land \varphi_{2}$.
For $i=1,2$, let $\theta_{i}$ be the restriction of $\theta$ to the free variables of $\varphi_{i}$.
Then $(M,\theta)$ satisfies $\varphi$ if and only if $(M,\theta_{1})$ satisfies $\varphi_{1}$ and $(M,\theta_{2})$ satisfies $\varphi_{2}$.
Finally, assume that $\varphi=\exists X\psi$.
Then $(M,\theta)$ satisfies $\varphi$ if and only if there is a subset, $Y\subseteq E$ such that $(M,\theta\cup\{(X,Y)\})$ satisfies $\psi$.

If $X$ is the only free variable in $\varphi$, then we may write $(M,X\mapsto Y)$ to mean $(M,\theta)$ where $\theta = \{(X,Y)\}$.
If $\varphi$ is a sentence, meaning it has no free variables, then we will say that $M$ satisfies $\varphi$ (instead of saying $(M,\emptyset)$ satisfies $\varphi$).
If $\varphi$ is a formula with the free variables $X_{i_{1}},\ldots, X_{i_{n}}$ and $M=(E,\mathcal{I})$ is a set-system, then we may say that the subsets $Y_{1},\ldots, Y_{n}\subseteq E$ \emph{satisfy} $\varphi$ to mean that $\varphi$ is satisfied by $M$ under the interpretation that takes each $X_{i_{j}}$ to $Y_{j}$.

We use various pieces of shorthand notation.
The formula $X=X'$ means $(X\subseteq X')\land (X'\subseteq X)$, and $X\ne X'$ means $\neg(X=X')$. We write $X\nsubseteq X'$ for $\neg(X\subseteq X')$.
If $\varphi_{1}$ and $\varphi_{2}$ are formulas with no variable bound in one and free in the other, then $\varphi_{1}\lor\varphi_{2}$ means $\neg(\neg \varphi_{1} \land \neg \varphi_{2})$, and $\varphi_{1}\to \varphi_{2}$ means $(\neg \varphi_{1})\lor \varphi_{2}$, while $\varphi_{1}\leftrightarrow\varphi_{2}$ means $(\varphi_{1}\to\varphi_{2})\land (\varphi_{2}\to\varphi_{1})$.
If $X$ is a free variable in the formula $\varphi$, then $\forall X\varphi$ is shorthand for $\neg(\exists X (\neg\varphi))$.
We call $\forall$ the universal quantifier.

Let $\varphi[X]$ be an \mso\ formula and let $k$ be a positive integer. 
Then there is an \mso\ formula that is satisfied by $M=(E,\mathcal{I})$ if and only if there exist exactly $k$ distinct subsets, $Y$, such that $(M,X\mapsto Y)$ satisfies $\varphi$.
The proof of this fact is routine, so we merely illustrate it with an explicit example:
If $k=2$, then the desired formula is 
\begin{linenomath*}
\begin{multline*}
\exists X'\exists X''(\varphi[X'] \land\varphi[X''] \land X'\ne X''\land\\
\forall X''' (\varphi[X''']\to (X'''=X'\lor X'''=X''))).
\end{multline*}
\end{linenomath*}
Here $X'$, $X''$, and $X'''$ stand for distinct variables that do not appear in $\varphi$.
A similar construction can be done for other values of $k$.
Henceforth, we use the notation $\exists_{k}X\varphi$ to stand for a formula that is satisfied if and only if there exist exactly $k$ distinct subsets, $Y$, such that $(M,X\mapsto Y)$ satisfies $\varphi$.

Let $\emp[X]$ stand for the formula $\exists_{1} X'(X'\subseteq X)$.
Note that $(M,X\mapsto Y)$ satisfies $\emp$ if and only if $Y$ has exactly one subset; that is, if and only if $Y=\emptyset$.
Now let $\sing[X]$ be the formula $\exists_{2} X'(X'\subseteq X)$.
Then $(M,X\mapsto Y)$ satisfies $\sing$ if and only if $|Y|=1$.

Let $M=(E,\mathcal{I})$ be a set-system.
We let $\formula{Basis}[X, X']$ be the formula \[X\subseteq X'\land\ind[X]\land \forall X''((X\subseteq X''\land X''\subseteq X')\to (X'' = X\lor \neg\ind[X''])).\]
Thus $(M,\theta)$ satisfies $\formula{Basis}[X,X']$ if and only if $\theta(X)$ is maximal with respect to being a subset of $\theta(X')$ and a member of $\mathcal{I}$.

For any positive integer $n$ we let $\formula{n-union}[X_{1},\ldots, X_{n},X]$ be the formula
\[
\forall X'(\sing[X'] \to (X'\subseteq X \leftrightarrow (X'\subseteq X_{1}\lor\cdots\lor X'\subseteq X_{n}))).
\]
Then $(M,\theta)$ satisfies $\formula{n-union}[X_{1},\ldots, X_{n},X]$ if and only if
\[\theta(X) = \theta(X_{1})\cup\cdots\cup \theta(X_{n}).\]
We let $\formula{RelDiff}[X_{1},X_{2},X]$ stand for \[\forall X'(\sing[X']\to (X'\subseteq X \leftrightarrow (X'\subseteq X_{1}\land X'\nsubseteq X_{2}))).\]
Therefore $(M,\theta)$ satisfies $\formula{RelDiff}[X_{1}, X_{2}, X]$ if and only if $\theta(X) = \theta(X_{1})-\theta(X_{2})$.

\subsection*{Matroids in monadic second-order logic.}
Now we state a series of results on the power of monadic second-order logic to express statements about matroids.
The next result is established in \cite{MR3803151}*{Section 2}.

\begin{proposition}
There is an \mso\ sentence, \formula{Matroid}, such that whenever $M=(E,\mathcal{I})$ is a set-system, $M$ satisfies $\formula{Matroid}$ if and only if $M$ is a matroid on the ground set $E$ and $\mathcal{I}$ is its family of independent sets.
\end{proposition}

If \formula{Matroid} is satisfied, then we can certainly say that a set is minimal with respect to being dependent, so we can characterise when a set is a circuit.
We can similarly characterise when a set is a basis.
Since a subset is a cocircuit if and only if it is minimal with respect to having a non-empty intersection with every basis \cite{MR2849819}*{Proposition 2.1.19}, we can characterise cocircuits.
For each matroid $N$, there is a sentence $\formula{Minor}_{N}$ such that $M$ satisfies $\formula{Minor}_{N}$ if and only if $M$ is a matroid with a minor isomorphic to $N$.
Thus a minor-closed class of matroids can be characterised by an \mso\ sentence, as long as that class has finitely many excluded minors.

\begin{proposition}
\label{prop_kseparators}
Let $k$ be a positive integer.
There is an \mso\ formula, $\formula{k-separation}[X]$, such that whenever $M=(E,\mathcal{I})$ is a set-system and $Y\subseteq E$ is a subset, $(M,X\mapsto Y)$ satisfies $\formula{k-separation}$ if and only if $M$ is a matroid and $(Y,E-Y)$ is a $k$\dash separation of $M$.
\end{proposition}

\begin{proof}
The formula $\formula{k-separation}$ will be a conjunction that has the sentence $\formula{Matroid}$ as one of its terms.
Henceforth we assume that $M$ is a matroid.
Recall that $(Y,E-Y)$ is a $k$\dash separation of $M$ if and only if $|Y|,|E-Y|\geq k$ and $r_{M}(Y)+r_{M}(E-Y)-r(M)<k$.
There is certainly a formula which will be satisfied if and only if $|Y|,|E-Y|\geq k$: the formula asserts that there exist $k$ pairwise-disjoint singleton sets contained in $Y$, 
and another $k$ pairwise-disjoint singleton sets not contained in $Y$.
Therefore we assume that $|Y|,|E-Y|\geq k$.

Now $r_{M}(Y)+r_{M}(E-Y)-r(M) = r_{M}(E-Y)-(r(M) - r_{M}(Y)) = i$ for some $i\in\{0,\ldots, k-1\}$ if and only if for every maximal independent subset, $B_{Y}$, of $Y$, and every basis $B$ that contains $B_{Y}$, we have $r_{M}(E-Y) = |B-Y|+i$.
This is true if and only if for every $B_{Y}$ satisfying $\formula{Basis}[B_{Y},Y]$ and every basis $B \supseteq B_{Y}$, 
there are $i$ pairwise-disjoint singleton sets, $Y_{1},\ldots, Y_{i}\nsubseteq Y$, such that when $Y'$ is equal to $B-Y$ and $Y''$ is equal to the union
\[Y'\cup Y_{1}\cup\cdots\cup Y_{i},\]
then $Y''$ is a maximally independent subset of $E-Y$.
It is clear that there is an \mso\ formula that is satisfied by $Y$ exactly when this condition holds.
The disjunction of these formulas for the $k$ values of $i$ is satisfied if and only if $r_{M}(Y)+r_{M}(E-Y)-r(M)<k$, so we are done. 
\end{proof}

Since a matroid is $n$\dash connected if and only if it has no $k$\dash separation with $k<n$, the next result follows immediately.

\begin{corollary}
Let $n$ be a positive integer.
There is an \mso\ sentence, $\formula{n-connected}$, such that whenever $M=(E,\mathcal{I})$ is a set-system, $M$ satisfies $\formula{n-connected}$ if and only if $M$ is an $n$\dash connected matroid.
\end{corollary}

Let $\varphi[X]$ be an \mso\ formula and
let $M=(E,\mathcal{I})$ be a set-system.
We define the family $\mathcal{F}_{\varphi}(M)$ to be
\[\{Y\colon Y\subseteq E, (M,X\mapsto Y)\ \text{satisfies}\ \varphi\}.\]

Let $E$ be a set, and let $\mathcal{F}$ be a collection of subsets of $E$.
We say $\mathcal{F}$ is a \emph{graphical family} of $E$ if
$1\leq |\{F\in \mathcal{F}\colon e\in F\}| \leq 2$ holds for every $e\in E$.
In this case, $G(\mathcal{F})$ is the graph with vertex-set $\mathcal{F}$, and $E$ as its edge-set, where an element $e\in E$ is incident with $F\in \mathcal{F}$ if and only if $e$ is in $F$.

\begin{proposition}
Let $\varphi[X]$ be an \mso\ formula and let $k$ be a non-negative integer.
There is an \mso\ formula, $\formula{Graphical}^{k}_{\varphi}[X_{1},\ldots, X_{k}]$, such that whenever $M=(E,\mathcal{I})$ is a set-system and $\theta$ is an interpretation, $(M,\theta)$ satisfies $\formula{Graphical}_{\varphi}^{k}$ if and only if  $\mathcal{F}_{\varphi}(M)\cup\{\theta(X_{1}),\ldots, \theta(X_{k})\}$ is a graphical family of $E$.
\end{proposition}

\begin{proof}
We set $\formula{Graphical}_{\varphi}^{k}$ to be
\begin{align*}
&\forall X\ \sing[X]\to\\
&\qquad(\exists_{1} X' (X\subseteq X'\land (\varphi[X']\lor X'=X_{1}\lor\cdots\lor X'=X_{k})))\\
&\qquad\lor (\exists_{2} X'' (X\subseteq X''\land (\varphi[X'']\lor X''=X_{1}\lor\cdots\lor X''=X_{k}))).\qedhere
\end{align*}
\end{proof}

In the following material, if $\varphi[X]$ is an \mso\ formula and $X_{1},\ldots, X_{k}$ are variables, then we use $\formula{Vertex}[X']$ to stand for the formula
\[
\varphi[X']\lor X'=X_{1}\lor\cdots\lor X'=X_{k}.
\]
Assuming that $\formula{Graphical}^{k}_{\varphi}[X_{1},\ldots, X_{k}]$ is satisfied, then $\formula{Vertex}[X']$ will be satisfied exactly when $\theta(X')$ is a member of the graphical family $\mathcal{F} = \mathcal{F}_{\varphi}(M)\cup\{\theta(X_{1}),\ldots, \theta(X_{k})\}$, and hence a vertex in the graph $G(\mathcal{F})$.

\begin{proposition}
Let $\varphi[X]$ be an \mso\ formula, and let $k$ be a non-negative integer.
There is an \mso\ formula, $\formula{Connected}_{\varphi}^{k}[X_{1},\ldots, X_{k}, X']$, such that whenever $M=(E,\mathcal{I})$ is a set-system and $\theta$ is an interpretation, $(M,\theta)$ satisfies $\formula{Connected}_{\varphi}^{k}$ if and only if $\mathcal{F} = \mathcal{F}_{\varphi}(M)\cup\{\theta(X_{1}),\ldots, \theta(X_{k})\}$ is a graphical family of $E$, and the subgraph $G(\mathcal{F})[\theta(X')]$ is connected.
\end{proposition}

\begin{proof}
We set $\formula{Connected}_{\varphi}^{k}$ to be a conjunction, with one of its terms equal to $\formula{Graphical}_{\varphi}^{k}[X_{1},\ldots, X_{k}]$.
So we henceforth assume that $\mathcal{F} = \mathcal{F}_{\varphi}(M)\cup\{\theta(X_{1}),\ldots, \theta(X_{k})\}$ is a graphical family.
The formula $\formula{Connected}_{\varphi}^{k}$ furthermore states that for every partition of $X'$ into two non-empty blocks, there exists $X''$ such that $\formula{Vertex}[X'']$ holds, and $X''$ contains at least one element from each block of the partition.
This condition can clearly be expressed in a \mso\ formula.
Then $(M,\theta)$ satisfies $\formula{Connected}_{\varphi}^{k}$ if and only if: for every partition of the edge subset $\theta(X')$ into two non-empty parts, there is a vertex of the graph $G(\mathcal{F})$ that is incident with at least one edge in both sides of the partition.
Clearly this condition is true if and only if $G(\mathcal{F})[\theta(X')]$ is connected.
\end{proof}

\begin{proposition}
Let $\varphi[X]$ be an \mso\ formula, and let $k$ be a non-negative integer.
There is an \mso\ formula, $\formula{Cycle}_{\varphi}^{k}[X_{1},\ldots, X_{k},X']$, such that whenever $M=(E,\mathcal{I})$ is a set-system and $\theta$ is an interpretation, $(M,\theta)$ satisfies $\formula{Cycle}_{\varphi}^{k}$ if and only if $\mathcal{F}=\mathcal{F}_{\varphi}(M)\cup\{\theta(X_{1}),\ldots, \theta(X_{k})\}$ is a graphical family of $E$, and the subgraph $G(\mathcal{F})[\theta(X')]$ is a cycle.
\end{proposition}

\begin{proof}
We construct $\formula{Cycle}_{\varphi}^{k}$ to be a conjunction containing $\formula{Connected}_{\varphi}^{k}$ as a term.
So henceforth we assume that $\mathcal{F}=\mathcal{F}_{\varphi}(M)\cup\{\theta(X_{1}),\ldots, \theta(X_{k})\}$ is a graphical family, and that $G(\mathcal{F})[\theta(X')]$ is connected.

The other term in the conjunction is itself the disjunction of two formulas.
The first formula is
\[
\sing[X']\land \exists_{1} X'' (X'\subseteq X'' \land \formula{Vertex}[X'']).
\]
This is satisfied if and only if $\theta(X')$ contains a single edge, and this edge is a loop in $G(\mathcal{F})$.
The second formula is the conjunction of
\[
\forall X'' (X''\subseteq X' \land \sing[X''] \to
\neg\exists_{1} X''' (X''\subseteq X''' \land \formula{Vertex}[X'''])
)
\]
and a formula saying that for every subset $X''$ satisfying $\formula{Vertex}[X'']$, $X''$ contains exactly two singleton subsets of $X'$, or no such subsets.
This conjunction asserts that $\theta(X')$ contains no loops of $G(\mathcal{F})$ and that every vertex of $G(\mathcal{F})$ that is incident with an edge in $\theta(X')$ is incident with exactly two such edges.

To summarise, $\formula{Cycle}_{\varphi}^{k}$ is satisfied if and only if: $\mathcal{F}$ is a graphical family and $G(\mathcal{F})[\theta(X')]$ is connected, and either $\theta(X')$ contains a single loop edge, or it contains no loops and every vertex is incident with either zero or two edges in $\theta(X')$.
Thus $\formula{Cycle}_{\varphi}^{k}$ is satisfied if and only if $G(\mathcal{F})[\theta(X')]$ is a cycle.
\end{proof}

\begin{proposition}
\label{prop_bicycleformula}
Let $\varphi[X]$ be an \mso\ formula, and let $k$ be a non-negative integer.
There is an \mso\ formula, $\formula{Bicycle}_{\varphi}^{k}[X_{1},\ldots, X_{k},X']$, such that whenever $M=(E,\mathcal{I})$ is a set-system and $\theta$ is an interpretation, $(M,\theta)$ satisfies $\formula{Bicycle}_{\varphi}^{k}$ if and only if $\mathcal{F} = \mathcal{F}_{\varphi}(M)\cup\{\theta(X_{1}),\ldots, \theta(X_{k})\}$ is a graphical family of $E$, and the subgraph $G(\mathcal{F})[\theta(X')]$ is a bicycle.
\end{proposition}

\begin{proof}
The bicycles are exactly the minimal sets of edges inducing a connected subgraph with at least two cycles.
So we set $\formula{Bicycle}_{\varphi}^{k}$ to say that $\formula{Graphical}_{\varphi}^{k}$ holds, and that the following statements are satisfied by $X'$, but no proper subset of $X'$: $\formula{Connected}_{\varphi}^{k}$ holds, and there exist two distinct subsets, $X'', X'''\subseteq X'$, such that $\formula{Cycle}_{\varphi}^{k}[X'']$ and $\formula{Cycle}_{\varphi}^{k}[X''']$ are both satisfied.
\end{proof}

\begin{lemma}
\label{lem_bicircularformula}
Let $\varphi[X]$ be an \mso\ formula and let $k$ be a non-negative integer.
There is an \mso\ formula, $\formula{Bicircular}_{\varphi}^{k}[X_{1},\ldots, X_{k}]$, such that whenever $M=(E,\mathcal{I})$ is a set-system and $\theta$ is an interpretation, $(M,\theta)$ satisfies $\formula{Bicircular}_{\varphi}^{k}$ if and only if $\mathcal{F} = \mathcal{F}_{\varphi}(M)\cup\{\theta(X_{1}),\ldots, \theta(X_{k})\}$ is a graphical family of $E$, and $M$ is the bicircular matroid $B(G(\mathcal{F}))$.
\end{lemma}

\begin{proof}
The desired formula is simply the conjunction of $\formula{Matroid}$, $\formula{Graphical}_{\varphi}^{k}[X_{1},\ldots, X_{k}]$, and the statement that for every subset $X'$, we have that $X'$ is a circuit of $M$ if and only if $\formula{Bicycle}_{\varphi}^{k}[X_{1},\ldots, X_{k},X']$ holds.
\end{proof}

\begin{lemma}
\label{connectedcase}
Let $\mathcal{M}$ be a class of matroids such that a matroid is in $\mathcal{M}$ if and only if each connected component is in $\mathcal{M}$.
If there is an \mso\ sentence that is satisfied exactly by the connected matroids in $\mathcal{M}$,
then there is a sentence that characterises $\mathcal{M}$.
\end{lemma}

\begin{proof}
We start by constructing the formula $\formula{Component}[X]$, relying on Proposition \ref{prop_kseparators}.
We set $\formula{Component}[X]$ to be the conjunction of \formula{Matroid} and the statement that $X$ is minimally non-empty with respect to being $1$\dash separating.
(We also allow $X$ to be empty in the case that $M=(E,\mathcal{I})$ has an empty ground-set.)
Thus $X$ satisfies \formula{Component} in $M=(E,\mathcal{I})$ if and only if $M$ is a matroid and $X$ is a connected component.

Let $\varphi$ be a sentence that is satisfied exactly by the connected matroids in $\mathcal{M}$.
Using standard methods, we can assume that $\varphi$ is in \emph{prenex normal form}.
That is, $\varphi$ has the form
\[
Q_{1}X_{i_{1}}\cdots Q_{n}X_{i_{n}}\ \omega,
\]
where each $Q_{i}$ is an existential or universal quantifier, and $\omega$ is a quantifier-free formula with $\{X_{i_{1}},\ldots, X_{i_{n}}\}$ as its set of variables.

Our new sentence is the conjunction of \formula{Matroid} and a sentence that we construct by starting with $\forall A\ \formula{Component}[A]\to$, and then modifying $\varphi$ by successively replacing $\exists X_{i_{j}}$ with
\[
\exists X_{i_{j}}\
X_{i_{j}}\subseteq A\ \land
\]
and replacing $\forall X_{i_{j}}$ with
\[
\forall X_{i_{j}}\
X_{i_{j}}\subseteq A \to.
\]
For example, if $\varphi$ were $\exists X_{1}\forall X_{2}\ \ind[X_{1}]\land (X_{1}\subseteq X_{2})$, then we would construct the sentence
\[
\forall A\ \formula{Component}[A] \to (\exists X_{1}\ X_{1}\subseteq A\land (\forall X_{2}\ X_{2}\subseteq A \to \ind[X_{1}]\land (X_{1}\subseteq X_{2}))).
\]
This new sentence will be satisfied if and only if $M$ is a matroid and each connected component satisfies $\varphi$.
\end{proof}

The conditions of Lemma \ref{connectedcase} apply when $\mathcal{M}$ is the class of bicircular matroids.
So we can now prove Theorem \ref{thm_main} by providing an \mso\dash characterisation of connected bicircular matroids.
We recall that a connected matroid $M$ is bicircular if and only if it is isomorphic to $U_{0,1}$, or if there is a graph $G$ such that $M=B(G)$.

\section{The 3-connected case}
\label{3-connected-case}

In this section we construct an \mso\dash characterisation of $3$\dash connected bicircular matroids. 

\begin{proposition}
\label{sammy}
Let $G$ be a connected graph with at least four vertices such that $M=B(G)$ is $3$\dash connected, and let $v$ be a vertex of $G$. 
Then $\st(v)$ is a non-separating cocircuit of $M$ if and only if $G-v$ is not a cycle and has no pendent edge.
\end{proposition}

\begin{proof}
Since $M$ is $3$\dash connected, Proposition \ref{conn} tells us that $G$ is $2$\dash connected.
If $G-v$ is a cycle or has a pendent edge, then $M\ba \st(v)$ contains a coloop.
As $G-v$ is connected and has at least three vertices, it has at least two edges, so $B(G-v) = M\backslash \st(v)$ is not connected.
Therefore $\st(v)$ is certainly not a non-separating cocircuit.

For the converse, we assume that $\st(v)$ is not a non-separating cocircuit.
We apply Proposition \ref{stars_are_cocircuits}.
If $\st(v)$ is not a cocircuit, then $G-v$ has no cycles, and therefore has a pendent edge.
Hence we can assume that $\st(v)$ is a cocircuit, so it must be a separating cocircuit, meaning that $M\ba \st(v)$ is not connected.
We apply Proposition \ref{conn} to $M\ba \st(v) = B(G-v)$, noting that $G-v$ is connected with at least three vertices.
We deduce that $G-v$ is either a cycle, or it has a pendent edge, as desired.
\end{proof} 

Let $G$ be a graph such that $B(G)$ is $3$\dash connected.
We say a vertex $v \in V(G)$ is \emph{committed} if $\st(v)$ is a non-separating cocircuit of $B(G)$. 
Otherwise $v$ is \emph{uncommitted}. 
By Proposition \ref{stars}, if $v$ is committed, then in every graph $H$ representing $B(G)$ there is a vertex $x$ with $\st(x)=\st(v)$. 
Thus if every vertex is committed, $B(G)$ is uniquely represented by $G$, up to relabelling of vertices. 
Proposition \ref{sammy} says that for $G$ sufficiently large, a vertex $v$ is committed if and only if $G-v$ properly contains a cycle and has no pendent edge. 
Thus to show that $v$ is uncommitted it suffices to show that $G-v$ is a cycle or has a pendent edge. 

\begin{lemma} \label{billybob}
Let $G$ be a connected graph with at least five vertices such that $M=B(G)$ is $3$\dash connected. 
Let $v$ be a vertex of $G$ such that $G-v$ has at most one cycle. 
Then $G$ has at most three uncommitted vertices. 
\end{lemma}

\begin{proof}
Note that since $M$ is $3$\dash connected, $G$ is $2$\dash connected and has minimum degree at least three.
Observe that either $G-v$ is a cycle, or it has a pendent edge.
Thus $v$ is uncommitted.

\begin{claim}
\label{sally} 
Let $x$ be a vertex of $G$ such that $G-x$ is a cycle.
Then every vertex is committed, except for $x$.
\end{claim}

\begin{proof}
Assume $G-x$ is a cycle, and note it has at least four vertices.
Then $x$ is uncommitted by Proposition \ref{sammy}.
Since $G$ has minimum degree at least three, each vertex of $G-x$ is adjacent to $x$.
Therefore $G$ is obtained from a wheel graph on at least five vertices by possibly adding parallel edges to the spokes, and possibly adding a loop at $x$, which is the hub of the wheel.
Now it follows that if $y$ is a vertex other than $x$, then $G-y$ is not a cycle, and has no pendent edges.
Therefore $y$ is committed.
\end{proof}

By Claim \ref{sally} we may now assume that there is no vertex whose deletion leaves just a cycle.

\begin{claim}
\label{wegotcycles}
For every vertex $x \neq v$, $G-x$ properly contains a cycle. 
\end{claim}

\begin{proof}
Since $G-v$ contains at most one cycle, and is not a cycle itself, it contains a leaf $u$.
As $G$ has minimum degree at least three, there are at least two edges joining $u$ and $v$.
So if $x$ is any vertex other than $u$ or $v$, then $G-x$ properly contains a two-edge cycle.
Suppose that $G-u$ does not properly contain a cycle.
Then $G-v$ does not have two leaves, or else we could apply a symmetric argument and deduce that $G-u$ has a two-edge cycle.
Hence $G-v$ is not a tree, so it contains a unique cycle, $C$.
Since $u$ is a leaf it is not in $C$, so $G-u$ properly contains $C$, a contradiction.
\end{proof}

\begin{claim} 
\label{claim_v_not_pendent}
For every vertex $x \in V(G)$, $v$ is not a leaf of $G-x$. 
\end{claim} 

\begin{proof}
Since $G-v$ is not a cycle, it contains at least one leaf.
There are at least two edges joining $v$ to any leaf of $G-v$.
So if $G-v$ has two leaves, then $v$ has degree at least two in $G-x$.
Therefore we assume that $G-v$ has exactly one leaf.
This means that $G-v$ is constructed by attaching a path to a cycle.
Since $G-v$ has at least four vertices, it has 
at least three vertices with degree at most two.
There is at least one edge joining $v$ to any such vertex.
It now follows that $v$ has degree at least two in $G-x$, as desired.
\end{proof}

By hypothesis, $G-v$ contains at most one cycle.
Assume that $G-v$ does not contain a two-edge cycle.
We claim that in this case, every vertex other than $v$ is committed.
Assume otherwise, and let $x \neq v$ be a vertex of $G$ such that $\st(x)$ is not a non-separating cocircuit.
Then either $G-x$ is a cycle or has a pendent edge, by Proposition \ref{sammy}.
Since $G-x$ is not a cycle by assumption, $G-x$ has a pendent edge. Let $w$ be the leaf incident to a pendent edge of $G-x$. 
By Claim \ref{claim_v_not_pendent}, $w \neq v$.
Since the minimum degree of $G$ is at least three, there are at least two edges joining $x$ and $w$.
This implies that $G-v$ contains a two-edge cycle, contrary to hypothesis.

Now we can assume that $G-v$ contains a two-edge cycle, $C$, and that $C$ is therefore the unique cycle in $G-v$.
We claim that if $x$ is a vertex not in $C$ and not equal to $v$, then $x$ is committed; for if not, then we can argue as in the previous paragraph that $G-x$ contains a leaf $w$, that $w\ne v$, and that $x$ and $w$ are joined by at least two edges.
But now $G-v$ contains two distinct cycles: $C$, and a pair of edges joining $x$ to $w$.
This contradiction completes the proof.
\end{proof}

\begin{definition}
Let $C^{*}$ be a cocircuit of the matroid $M$.
Assume that $M\ba C^{*}$ has exactly one connected component, $D$, with size greater than one.
Assume also that whenever $x$ is a coloop of $M\ba C^{*}$, there is a rank\dash $2$ clonal class, $F$, of $M$ contained in $C^{*}$, and a circuit, $C$, of $M$ such that
\begin{enumerate}[label=\textup{(\roman*)}]
\item $x\in C$,
\item $C\cap C^{*}$ is a $2$-element subset of $F$, and
\item $C\cap D\ne \emptyset$.
\end{enumerate}
In this case we say $C^{*}$ is a \emph{good cocircuit}.
\end{definition}

Note that every non-separating cocircuit of $M$ is vacuously a good cocircuit.

\begin{lemma}
\label{lem_goodcocircuits}
Let $G$ be a connected graph with at least five vertices such that $M=B(G)$ is $3$\dash connected.
Assume that $G-v$ has at least two cycles for every vertex $v$.
Then every vertex star of $G$ is a good cocircuit of $M$, and every good cocircuit of $M$ is a vertex star of $G$.
\end{lemma}

\begin{proof}
By Proposition \ref{conn}, $G$ is $2$\dash connected with minimum degree at least three, and no vertex is incident with more than one loop.
Let $v$ be a vertex. 
By Proposition \ref{stars_are_cocircuits}, $\st(v)$ is a cocircuit of $M$.
Because $G-v$ is connected and contains at least two cycles, $B(G-v) = M\ba \st(v)$ contains a circuit with at least two elements.
Therefore $M\ba \st(v)$ has a component with at least two elements.
Proposition \ref{prop_onecomp} tells us that $M\ba \st(v)$ has at most one such component, so it has exactly one.
Let $D$ be this component; any two edges in $G[D]$ are contained in a common bicycle.

Let $x$ be a coloop of $M\ba \st(v)$.
Note that $x$ cannot be a loop in $G$, for otherwise we can find a path that joins the vertex incident with $x$ to a vertex in $G[D]$.
As $G[D]$ is connected and contains a cycle, this implies that there is a circuit of $M\ba\st(v)$ containing $x$, which is impossible as $x$ is a coloop in this matroid.
Let $P'$ be a path in $G-v$ with vertex sequence $v_{0},\ldots, v_{m}$, where $P'$ is chosen to be as short as possible subject to $v_{0}$ being in $G[D]$, and $x$ joining $v_{m-1}$ to $v_{m}$.
We extend $P'$ as far as possible to a path, $P$, of $G-v$ with the vertex sequence $v_{0},\ldots, v_{m},v_{m+1},\ldots, v_{n}$.
We claim that $v_{n}$ is a leaf in $G-v$.
Assume otherwise.
Let $e$ be an edge of $G-v$ incident with $v_{n}$ but not in $P$.
If $e$ were incident with a vertex in $D$, then we could find a bicycle of $G-v$ containing all the edges in $P$, and this contradicts the fact that $x$ is a coloop in $M\ba \st(v)$.
So $e$ is either a loop, or it joins $v_{n}$ to a vertex in
$v_{0},\ldots, v_{n-1}$, for otherwise we can extend $P$ by the edge $e$, and contradict our choice.
Assume that $e$ joins $v_{n}$ to $v_{i}$, for some $i\in\{0,\ldots, n\}$ (note that this includes the case that $e$ is a loop incident with $v_{n}$).
Therefore $v_{i},\ldots, v_{n},v_{i}$ is the vertex sequence of a cycle, and $v_{0},\ldots, v_{i}$ is the vertex sequence of a path joining this cycle to a vertex in $G[D]$.
As $G[D]$ is connected and contains a cycle, we can now find a handcuff of $G-v$ that contains $x$.
This contradicts the fact that $x$ is a coloop in $M\ba\st(v)=B(G-v)$.
Therefore $v_{n}$ is a leaf of $G-v$, as claimed.

Since $G$ has minimum degree at least three, there are at least two parallel edges joining $v$ to $v_{n}$.
Let $F$ be the parallel class containing these two edges.
It follows from Proposition \ref{prop_bicircularclones} that $F$ is a rank\dash $2$ clonal class of $M$.
Any two edges in $F$ form a cycle of $G$, and $P$ is a path joining this cycle to a vertex in $G[D]$.
As $G[D]$ is connected and contains a cycle, we can find a handcuff of $G$, hence a circuit of $M$, consisting of two edges from $F$, the edges in $P$ (including $x$), and a subset of $D$.
This is enough to certify that $\st(v)$ is a good cocircuit of $M$.

For the converse, we let $C^{*}$ be a good cocircuit of $M$.
Let $D$ be the unique component of $M\ba C^{*}$ that contains more than one element.
Note that any element not in $C^{*}\cup D$ is a coloop of $M\ba C^{*}$.
If $C^{*}$ does not contain a bond of $G$, then it is the complement of a spanning tree in $G$.
But this would imply that every component of $M\ba C^{*}$ is a coloop, which is impossible as $D$ contains at least two elements.
Therefore we let $B$ be a bond contained in $C^{*}$.
Let $G_{1}$ and $G_{2}$ be the connected components of $G\ba B$.
Without loss of generality, we can assume that $C^{*}$ is the disjoint union of $A$ and $B$, where $A$ is a minimal set of edges in $G_{2}$ such that $G_{2}\ba A$ is a tree.
This means that $D$ is contained in the edge-set of $G_{1}$.

If $G_{2}$ contains a single vertex, $u$, then $A$ is the set of loops incidence with $u$, and $B$ is the set of non-loop edges incident with $u$.
In this case $C^{*}=\st(u)$, and we have nothing left to prove.
Therefore we can let $x$ be an edge in $G_{2}\ba A$.
Thus $x$ is a coloop in $M\ba C^{*}$.
Since $C^{*}$ is a good cocircuit, there is a rank\dash $2$ clonal class, $F$, and a circuit, $C$, of $M$, containing $x$, such that $C\cap C^{*}$ is a $2$-element subset of $F$, and $C$ contains at least one element of $D$.
Thus $C$ is a bicycle of $G$.
We let $e$ and $f$ be the two elements in $C\cap C^{*}$.
Then $e$ and $f$ belong to $F$, so they are clones.
From Proposition \ref{prop_bicircularclones}, we see that $e$ and $f$ are parallel edges of $G$.
As $e$ and $f$ are in $C^{*}=A\cup B$, it is obvious that both edges are in $A$, or both are in $B$.
Assume $e,f\in A$.
Then $C\cap B=\emptyset$, since otherwise $C \cap C^*$ is not a $2$\dash element subset of $F$.
But $C$ contains an edge of $G_{2}$ (namely $x$), as well as edges of $G_{1}$ (since $C$ contains at least one edge in $D$).
This leads to a contradiction, as $G[C]$ is a connected graph.
Therefore $e$ and $f$ are both in $B$.
Let $u$ be the unique vertex in $G_{2}$ that is incident with $e$ and $f$.
The edges of $C$ that are in $G_{2}$ form a forest, since $C$ contains no edges in $A$, and $G_{2}\ba A$ is a tree.
There is at least one edge of $C$ that is in $G_{2}$ (namely $x$), so this forest contains at least two leaves.
We let $v$ be such a leaf that is distinct from $u$.
But then $v$ has degree one in $G[C]$, since it is not incident with any edge in $C\cap B$.
However, $G[C]$ has no degree-one vertices, so we have a contradiction.
\end{proof}

From Lemma \ref{lem_goodcocircuits} we see that if $M=B(G)$ is $3$\dash connected, $G$ has at least five vertices, and $G-v$ always has at least two cycles, then $G$ is the only graph that represents $M$ as a bicircular matroid.

A \emph{rooted matroid} is a pair $(M,L)$ consisting of a matroid $M$ together with a distinguished subset $L$ of its elements. 
A rooted matroid $(M,L)$ is \emph{bicircular} if there is a graph $G$ so that $M=B(G)$ and every element in $L$ is a loop of $G$.

\begin{theorem}
\label{lem_3concase}
There is an \mso\ formula, $\formula{BicircularLoops}[X]$, such that whenever $M=(E,\mathcal{I})$ is a set-system and $L\subseteq E$ is a subset, $(M,X\mapsto L)$ satisfies $\formula{BicircularLoops}$ if and only if $(M,L)$ is a $3$\dash connected rooted bicircular matroid. 
\end{theorem}

\begin{proof}
The formula $\formula{BicircularLoops}[X]$ will be a conjunction containing \formula{3-Connected} as a term.
Therefore we may as well assume from this point forward that $M$ is a $3$\dash connected matroid.

A graph, $G$, satisfies exactly one of the following three cases.
\begin{enumerate}[label=\textup{(\roman*)}]
\item $G$ has at most four vertices,
\item $G$ has at least five vertices, and a vertex $v$, such that $G-v$ has at most one cycle, or
\item $G$ has at least five vertices, and $G-v$ has  two cycles for any vertex $v$.
\end{enumerate}
In each case, we develop a formula which will be satisfied by $M$ and $L$ exactly when $M$ is the bicircular matroid of $G$ and each element in $L$ is a loop in $G$.
We will then simply take the disjunction of these formulas, and we will be done.

For case (i), we will rely on the formula $\formula{Bicircular}_{\varphi}^{i}$ from Lemma \ref{lem_bicircularformula}.
In this case, we want $\varphi$ to be a formula with a single free variable that is not satisfied by any subset of $E$.
For specificity, let $\varphi[X']$ be $\ind[X']\land \neg\ind[X']$.
Now $M$ is the bicircular matroid of a graph with $i$ vertices (where $i\in \{1,2,3,4\}$) if and only if there exist sets $Y_{1},\ldots, Y_{i}$ such that $(M,\theta)$ satisfies $\formula{Bicircular}_{\varphi}^{i}[X_{1},\ldots, X_{i}]$ under the interpretation that takes each variable $X_{j}$ to $Y_{j}$.
So the formula we want for a particular value of $i$ is the conjunction of
\[\exists X_{1}\cdots\exists X_{i} \formula{Bicircular}_{\varphi}^{i}[X_{1},\ldots, X_{i}]
\]
with a formula asserting that each element in $X$ is a loop in the bicircular representation.
We can accomplish the last part by asserting that for any singleton subset $X''\subseteq X$, there is exactly one subset $X'''$ such that $X''\subseteq X'''$ and $\formula{Vertex}[X''']$ holds.
We then take the disjunction of the four formulas, one for each value of $i$.

Now we move to case (ii).
Let $C^{*}$ be a cocircuit of the matroid $M$.
Then $C^{*}$ is a non-separating cocircuit if and only if, for every pair of distinct elements, $e,f\notin C^{*}$, there is a circuit of $M$ that contains $e$ and $f$ and has an empty intersection with $C^{*}$.
From this it follows that there is an \mso\ formula, $\formula{NonSepCocircuit}[X']$, such that $(M,X'\mapsto Y')$ satisfies $\formula{NonSepCocircuit}$ if and only if $Y'$ is a non-separating cocircuit of $M$.
Let $\varphi$ stand for $\formula{NonSepCocircuit}$.

Lemma \ref{billybob} implies that in case (ii), $M$ is bicircular if and only if $M$ satisfies $\formula{Bicircular}_{\varphi}^{0}$, or there exist sets $Y_{1},\ldots, Y_{i}$ for $i\in\{1,2,3\}$ such that $(M,\theta)$ satisfies $\formula{Bicircular}_{\varphi}^{i}[X_{1},\ldots, X_{i}]$ where $\theta$ takes each $X_{j}$ to $Y_{j}$.
Moreover, if $\mathcal{F}(\varphi)$ is a graphical family, then $e\in E$ is a loop in $G(\mathcal{F}(\varphi))$ if and only if there is exactly one set that contains $e$ and satisfies $\varphi$.
Similarly, if $\mathcal{F} = \mathcal{F}(\varphi)\cup\{Y_{1},\ldots, Y_{i}\}$ is a graphical family, then $e\in E$ is a loop in $G(\mathcal{F})$ if and only if there is exactly one set that contains $e$ and either satisfies $\varphi$, or is equal to one of $Y_{1},\ldots, Y_{i}$.
Since this is clearly expressible in \mso, the formula we want for case (ii) can be constructed as the disjunction of $\formula{Bicircular}_{\varphi}^{0}$ with three other formulas, one for each value of $i$.

Case (iii) will fall to exactly the same argument, except that we must now use Lemma \ref{lem_goodcocircuits}, and show that there is an \mso\ formula, $\formula{GoodCocircuit}[X']$ such that $(M,X'\mapsto Y')$ satisfies $\formula{GoodCocircuit}$ if and only if $Y'$ is a good cocircuit in the matroid $M$.
To this end, we make the following observations.
Let $M$ be a matroid on the ground set $E$, and let $C^{*}$ be a cocircuit of $M$.
Then $C$ is a component of $M\ba C^{*}$ if and only if for every pair of distinct elements in $C$, there is a circuit of $M|C$ that contains both of them, and $C$ is a maximal subset of $E-C^{*}$ with respect to this property.
Thus we have an \mso-characterisation of components of $M\ba C^{*}$, and we can check that there is at most one such component that is not a singleton set.
Next we observe that $x$ is a coloop of $M\ba C^{*}$ if and only if every circuit of $M$ that contains $x$ also contains an element of $C^{*}$.
A set, $F$, is a flat of $M$ if and only if $I\cup z$ is independent, for every independent subset $I\subseteq F$ and every element $z\in E-F$.
Moreover, we can certainly characterise when $F$ is a union of circuits.
Thus cyclic flats have an \mso\dash characterisation.
From this we see that we can characterise when two elements are clones, by stating that a cyclic flat contains one if and only if it contains the other.
Now it follows easily that there is an \mso-characterisation of rank\dash $2$ clonal classes.
From this point, it is routine to verify that there is an \mso-characterisation of good cocircuits, as we claimed, and this completes the proof of the lemma.
\end{proof}

We note the following consequence of Theorem \ref{lem_3concase}: $(M,X\mapsto \emptyset)$ satisfies $\formula{BicircularLoops}$ if and only if $M$ is a $3$\dash connected bicircular matroid.
So at this point in our arguments, we have established an \mso-characterisation of $3$\dash connected bicircular matroids.

\section{Reducing to the 3-connected case} 
\label{Reduceto3}

This section is dedicated to using our \mso\dash characterisation of $3$\dash connected bicircular matroids to leverage a characterisation of the entire class.
We need to develop several technical tools.

\subsection*{Decompositions of connected matroids.}
Cunningham and Edmonds showed that a connected matroid can be decomposed by a canonical tree that displays all of its $2$\dash separations.
We now explain this work, following \cite{MR2849819}*{Section 8.3}.
Let $T$ be a tree and assume each node of $T$ is a matroid.
Let the edges of $T$ be labelled with the distinct elements $e_{1},\ldots, e_{t}$.
We impose the following conditions:
\begin{enumerate}[label=\textup{(\roman*)}]
\item if $T$ has more than one node, then each node $N$ satisfies $|E(N)|\geq 3$,
\item if nodes $N$ and $N'$ are not adjacent in $T$, then $E(N)\cap E(N')=\emptyset$, and
\item if nodes $N$ and $N'$ are joined by an edge, $e_{i}$, then $E(N)\cap E(N')=\{e_{i}\}$, where $e_{i}$ is not a separator in $N$ or $N'$
(in this case, we say that $e_{i}$ is a \emph{basepoint} of $N$ and $N'$).
\end{enumerate}
When these conditions hold, we say that $T$ is a \emph{matroid-labelled tree}.
We recursively describe a matroid, $M(T)$, associated with this tree.
If $T$ contains only a single node, $N$, then $M(T)=N$.
Otherwise assume that $T$ contains an edge $e$, joining nodes $N$ and $N'$.
We let $T/e$ be the tree obtained by contracting $e$, where the identified node now corresponds to the $2$\dash sum of $N$ and $N'$ along the basepoint $e$.
We recursively define $M(T)$ to be $M(T/e)$.
Since the $2$\dash sum operation is associative \cite{MR2849819}*{Proposition 7.1.23}, $M(T)$ is well-defined, and does not depend on the order in which we contract edges.
Note that when $N$ is a node, the intersection of $E(M(T))$ and $E(N)$ may well be empty (if every element in $E(N)$ is a basepoint).
However, if $N$ is a leaf of $T$, then the requirement that $E(N)$ has at least three elements means that $E(N)$ contains at least two elements in $E(M(T))$.
If $T^{*}$ is the tree obtained from $T$ by replacing each node with its dual, then $M(T^{*})=(M(T))^{*}$.

Let $T$ be a matroid-labelled tree, and let $M$ be $M(T)$.
Let $T'$ be a subgraph of $T$.
We define $E(T')$ to be $\cup_{N\in V(T')} E(N)\cap E(M)$.
Thus $E(T')$ contains no basepoint elements.

In the next theorem, when we refer to a circuit or a cocircuit, we mean the matroids in which the entire ground set is a circuit or a cocircuit, respectively.
The next result is due to Cunningham and Edmonds  \cite{MR586989} (see \cite{MR2849819}*{Theorem 8.3.10}).

\begin{theorem}
\label{CunninghamEdmonds}
Let $M$ be a connected matroid.
There is a unique (up to relabelling edges) matroid-labelled tree, $T$, with the following properties: $M=M(T)$, every node of $T$ is a $3$\dash connected matroid, a circuit, or a cocircuit, and furthermore, no edge of $T$ joins a circuit to a circuit, or a cocircuit to a cocircuit.
\end{theorem}

We refer to the unique matroid-labelled tree from Theorem \ref{CunninghamEdmonds} as the \emph{canonical decomposition tree} of $M$.
Let $T$ be the canonical decomposition tree for the connected matroid $M$.
The \emph{$3$\dash connected components} of $M$ (and $T$) are the nodes of $T$ with rank and corank at least two.
Note that any such matroid is necessarily simple, cosimple, and $3$\dash connected, and that a node of $T$ can be $3$\dash connected without being a $3$\dash connected component (since not every $3$\dash connected matroid has rank and corank at least two).
Let $e$ be an edge of $T$, and let the two connected components of $T\ba e$ be $T_{A}$ and $T_{B}$.
Let $A$ be $E(T_{A})$ and let $B$ be $E(T_{B})$.
Then $(A,B)$ is a partition of $E(M)$, and we say that $A$, $B$, and $(A,B)$ are \emph{displayed} by the edge $e$.
Now let $N$ be a node in $T$ and let $(A,B)$ be a partition of $E(M)$ such that whenever $T'$ is a connected component in the forest $T-N$, the set $E(T')$ is contained in one of $A$ or $B$.
Then we say that $A$, $B$, and $(A,B)$ are \emph{displayed} by the node $N$.
Note that a partition displayed by an edge is also displayed by any node incident with that edge.
The next result is Proposition 8.3.16 in \cite{MR2849819}.

\begin{proposition}
\label{treedisplayed}
Let $T$ be the canonical decomposition tree of a connected matroid $M$.
Let $(A,B)$ be a partition of $E(M)$ such that $|A|,|B|\geq 2$.
Then $(A,B)$ is a $2$\dash separation of $M$ if and only if it is displayed by an edge of $T$ or if it is a displayed by a circuit node or a cocircuit node.
\end{proposition}

Let $M$ be a connected matroid, and let $A$ be a $2$\dash separating set.
Let $B$ be the complement $E(M)-A$.
If $X\subseteq B$ is maximal amongst the $2$\dash separating non-empty proper subsets of $B$, then we say that $X$ is a \emph{wedge} (relative to $A$).
Note that a wedge may be a singleton set.

\begin{proposition}
\label{vertexwedges}
Let $T$ be the canonical decomposition tree of the connected matroid $M$.
Let $N$ be a circuit or cocircuit node of $T$.
Let $(A,B)$ be a $2$\dash separation of $M$ displayed by $N$, where $T_{1},\ldots, T_{n}$ are the components of $T-N$ satisfying $E(T_{i})\subseteq B$.
Assume $n\geq 2$.
The wedges relative to $A$ are the sets of the form $B-E(T_{i})$ or $B-b$, where $b$ is in $B\cap E(N)$.
\end{proposition}

\begin{proof}
Certainly $B-b$ is $2$\dash separating when $b$ is in $B\cap E(N)$, for it is displayed by $N$.
So $B-b$ is obviously a wedge.
Similarly $B-E(T_{i})$ is $2$\dash separating.
Assume that $B-E(T_{i})$ is not a wedge, so that it is properly contained in a wedge $X$.
From $n\geq 2$ we deduce $|X|\geq 2$, so Proposition \ref{treedisplayed} implies that $X$ is displayed by an edge, or by a circuit or cocircuit node.
But $X$ contains at least some elements of $E(T_{i})$.
On the other hand, $E(T_{i})\nsubseteq X$, for otherwise $X=B$.
Thus $X$ is displayed by an edge or node in $T_{i}$.
As $n\geq 2$, we see that $X$ contains $E(T_{j})$ for some $j\ne i$.
We can now deduce that $X$ must contain $A$, which is impossible.
So $B-E(T_{i})$ is a wedge, as desired.

Now let $X$ be a wedge relative to $A$ that is not equal to $B-E(T_{i})$ for any $i$ or to $B-b$ for any $b\in B\cap E(N)$.
Since $X$ is not properly contained in a wedge, it follows that $X$ is not a singleton set.
Therefore we can apply Proposition \ref{treedisplayed}.
Furthermore, $X$ must contain $B\cap E(N)$.
If $X$ is displayed by an edge incident with $N$, then it contains either $A$ or $B$, which is impossible.
If $X$ is displayed by any edge or node of $T_{i}$, then either it contains $A$, or it is contained in $B-E(T_{j})$ when $j\ne i$.
In either case we have a contradiction.
Since $X$ does not contain $B$, the only remaining possibility is that $X$ is displayed by $N$.
But $X$ is not contained in $B-E(T_{i})$ for any $i$, so we are forced to conclude that $X=B$, a contradiction.
\end{proof}

When $X$ and $Y$ are disjoint subsets of $E(M)$, we say that $X$ and $Y$ are \emph{skew} (in $M$) if no circuit of $M|(X\cup Y)$ contains elements of both $X$ and $Y$.
If $X$ and $Y$ are skew in $M^{*}$, then we say that they are \emph{coskew} in $M$.

\begin{proposition}
\label{characterisingcomponents}
Let $T$ be the canonical decomposition tree of the connected matroid $M$.
Let $(A,B)$ be a $2$\dash separation of $M$.
Assume that the following conditions hold:
\begin{enumerate}[label=\textup{(\roman*)}]
\item $\cl(A)\cap B = \emptyset = \cl^{*}(A)\cap B$,
\item any two distinct wedges relative to $A$ are disjoint, skew, and coskew.
\end{enumerate}
Then $(A,B)$ is displayed by an edge $e$ in $T$, and if $T_{B}$ is the component of $T\ba e$ such that $B = E(T_{B})$, then the node of $T_{B}$ incident with $e$ is a $3$\dash connected component of $M$.
\end{proposition}

\begin{proof}
Let us assume for a contradiction that $(A,B)$ is not displayed by an edge.
Proposition \ref{treedisplayed} implies that it is displayed by a circuit or cocircuit node $N$.
Let $T_{1},\ldots, T_{n}$ be the components of $T-N$ such that $E(T_{i})\subseteq B$ for each $i$.
For each $i$ let $e_{i}$ be the edge of $T$ joining $N$ to $T_{i}$.

If $n=0$, then $B\subseteq E(N)$, so $B$ is a subset of a parallel or series class.
Since $M$ is connected, it is simple to verify that $B$ is contained in $\cl(A)$ or $\cl^{*}(A)$, contrary to hypothesis.
So $n\geq 1$.

Assume that $n=1$.
If $E(N)$ contains no element of $B$, then $(A,B)$ is displayed by $e_{1}$, contrary to our hypothesis.
So let $b$ be an element in $B\cap E(N)$.
Assume that $A\cap E(N)$ contains an element $a$.
Then $\{a,b\}$ is a circuit or a cocircuit in $N$, and also in $M$.
We see that $b$ is in either $\cl(A)$ or $\cl^{*}(A)$, contradicting condition (i).
Therefore $A\cap E(N) = \emptyset$.
This means that $N$ is not a leaf in $T$, for otherwise $A$ is empty.
We let $e$ be any edge other than $e_{1}$ incident with $N$, and let $N'$ be the other node incident with $e$.
Then $M$ can be expressed as $M_{1}\oplus_{2} M_{2}$, where the $2$\dash sum is taken along the basepoint $e$, and $e$ displays the $2$\dash separation $(E(M_{1})-e, E(M_{2})-e)$.
We assume that $b$ is in $E(M_{2})$.
If $N$ is a cocircuit, then $\{b,e\}$ is a circuit of $M_{2}$.
We let $C$ be a circuit of $M_{1}$ that contains $e$.
Then $(C-e)\cup b$ is a circuit of $M$, so $b$ is in $\cl(A)$, and we have a contradiction to condition (i).
If $N$ is a circuit node, then we similarly deduce that $b$ is in $\cl^{*}(A)$.
Therefore we have to conclude that $n\geq 2$.

Proposition \ref{vertexwedges} implies that $B-E(T_{i})$ is a wedge relative to $A$ for each $i$.
In particular, $B-E(T_{1})$ and $B-E(T_{2})$ are distinct wedges, and are therefore disjoint.
From this we deduce $E(N)\cap B=\emptyset$ and $n=2$, for otherwise $(B-E(T_{1}))\cap (B-E(T_{2}))$ contains an element in $E(N)\cap B$ or $E(T_{3})$.
Now it follows that the wedges relative to $A$ are $E(T_{1})$ and $E(T_{2})$.
We express $M$ as $M_{0}\oplus_{2} M_{1}\oplus_{2} M_{2}$, where $E(M_{0}) = A\cup\{e_{1},e_{2}\}$ and
$E(M_{i}) = E(T_{i})\cup e_{i}$ for $i=1,2$.
Assume $N$ is a cocircuit node.
Note that $\{e_{1},e_{2}\}$ is a circuit in $M_{0}$.
For $i=1,2$, we let $C_{i}$ be a circuit of $M_{i}$ that contains $e_{i}$.
Now $(C_{1}-e_{1})\cup (C_{2}-e_{2})$ is a circuit of $M$, but this contradicts the condition that $E(T_{1})$ and $E(T_{2})$ are skew.
If $N$ is a circuit node, then we can deduce that $E(T_{1})$ and $E(T_{2})$ are not coskew.
This contradiction shows that $(A,B)$ is displayed by an edge, $e$, of $T$.

Let $T_{B}$ be the component of $T\backslash e$ such that $B=E(T_{B})$, and let $N$ be the node of $T_{B}$ that is incident with $e$.
We can complete the proof by showing that $N$ has rank and corank at least two.
Express $M$ as the $2$\dash sum $M_{A}\oplus_{2} M_{B}$ along the basepoint $e$, where $E(M_{B})=E(T_{B})\cup e$.
Assume $r(N)<2$ or $r^{*}(N)<2$, so that $N$ is either a circuit or a cocircuit.
By duality, we can assume it is a cocircuit.
Assume that $B\cap E(N)$ contains an element $b$, which is therefore parallel to $e$ in $N$.
If $C$ is a circuit of $M_{A}$ containing $e$, then $(C-e)\cup b$ is a circuit of $M$, implying that $b$ is in $\cl(A)$.
As this is a contradiction, we deduce that $B\cap E(N)=\emptyset$, so every element in $E(N)$ is a basepoint.
Since $|E(N)|\geq 3$, it therefore follows that $N$ has degree at least three in $T$.
Let $e_{1},\ldots, e_{n}$ be the edges of $T_{B}$ incident with $N$, and note $n\geq 2$.
For each $i$, let $T_{i}$ be the component of $T\ba e_{i}$ that does not contain $N$, so that $B$ is the disjoint union of $E(T_{1}),\ldots, E(T_{n})$.

If $n\geq 3$, then Proposition \ref{vertexwedges} implies that $B$ contains distinct non-disjoint wedges.
Therefore $n=2$ and $E(T_{1})$ and $E(T_{2})$ are the wedges relative to $A$.
Furthermore, there is a circuit of $M$ contained in $E(T_{1})\cup E(T_{2})$ that contains elements from both $E(T_{1})$ and $E(T_{2})$.
Since this contradicts condition (ii), we have completed the proof.
\end{proof}

A converse result also holds.

\begin{proposition}
\label{conversecomponents}
Let $T$ be the canonical decomposition tree of the connected matroid $M$.
Let $N$ be a $3$\dash connected component of $M$, and let $e$ be an edge of $T$ incident with $N$.
Let $T_{B}$ be the component of $T\ba e$ that contains $N$, and let $T_{A}$ be the other component.
Set $(A,B)$ to be $(E(T_{A}), E(T_{B}))$.
Let $e_{1},\ldots, e_{n}$ be the edges of $T_{B}$ incident with $N$, and for each $i$ let $T_{i}$ be the component of $T_{B}\ba e_{i}$ that does not contain $N$.
The wedges relative to $A$ are $E(T_{1}),\ldots, E(T_{n})$ and the singleton subsets of $E(N)-\{e,e_{1},\ldots, e_{n}\}$.
Moreover,
\begin{enumerate}[label=\textup{(\roman*)}]
\item $\cl(A)\cap B = \emptyset = \cl^{*}(A)\cap B$,
\item any two distinct wedges relative to $A$ are disjoint, skew, and coskew.
\end{enumerate}
\end{proposition}

\begin{proof}
Let $x$ be an element of $E(N)-\{e,e_{1},\ldots, e_{t}\}$.
Then $\{x\}$ is a $2$\dash separating subset of $B$.
Assume that $X$ is a $2$\dash separating proper subset of $B$ that properly contains $\{x\}$.
Then $|X|\geq 2$, and Proposition \ref{treedisplayed} tells us that $X$ is displayed by an edge of $T$, or by a circuit or cocircuit node in $T$.
Because $N$ is not a circuit or cocircuit node it cannot display any $2$\dash separations.
We observe that $X$ cannot be displayed by an edge or a node in $T_{B}$, for then $x\in X$ implies $A\subseteq X$, which is not true.
On the other hand, if $X$ is displayed by an edge or a node not in $T_{B}$, then $X\supseteq B$, which is also not true.
So we have shown that $\{x\}$ is a wedge relative to $A$.

Now assume that $E(T_{i})$ is properly contained in $X$, a $2$\dash separating proper subset of $B$.
Then $X$ is displayed by an edge in $T$ or by a circuit or cocircuit node.
But such an edge or node cannot be in $T_{i}$, for then $X$ could not properly contain $E(T_{i})$.
Similarly, if $X$ is displayed by $e_{i}$, then $X=E(T_{i})$, a contradiction.
If $X$ is displayed by any other edge or node of $T_{B}$, then $X$ contains $A$, which is impossible.
If $X$ is displayed by an edge or node not in $T_{B}$, then it contains $B$, which is also impossible.

We have shown that  $E(T_{1}),\ldots, E(T_{n})$ and the singleton subsets of $E(N)-\{e,e_{1},\ldots, e_{n}\}$ are all wedges relative to $A$.
These sets partition $B$.
Let $X$ be a wedge that is not equal to any of these sets.
Since no wedge can properly contain another, we see that $X$ contains no element from $E(N)-\{e,e_{1},\ldots, e_{n}\}$ and that $X$ contains at least two elements.
This implies $n\geq 1$, and by relabelling as appropriate, we can assume that $X$ contains a non-empty proper subset of $E(T_{1})$.
Then $X$ must be displayed by an edge or node in $T_{1}$, for otherwise $X$ contains all of $E(T_{1})$.
But now $X$ is either a subset of $E(T_{1})$, or it contains $A$.
In either case, we have a contradiction, so the wedges relative to $A$ are exactly $E(T_{1}),\ldots, E(T_{n})$ and the singleton subsets of $E(N)-\{e,e_{1},\ldots, e_{n}\}$.

For the second part of the proof, we express $M$ as the $2$\dash sum of $M_{A}$ and $M_{B}$ along the basepoint $e$, where $E(M_{A}) = A\cup e$ and $E(M_{B}) = B\cup e$.
Assume that $b$ is in $\cl(A)\cap B$.
Then there is a circuit of $M$ contained in $A\cup b$ that contains $b$.
This implies that $\{b,e\}$ is a circuit in $M_{B}$ and hence in $N$.
But this is impossible as $N$ is simple and cosimple.
We similarly derive a contradiction if there is an element in $\cl^{*}(A)\cap B$.
Hence condition (i) holds.

The first part of the proof shows that distinct wedges relative to $A$ are disjoint.
Since $N$ is simple and cosimple, any pair of singleton wedges are skew and coskew.
If the wedges $E(T_{i})$ and $\{x\}$ fail to be skew and coskew, then as in the previous paragraph, we deduce that $\{e_{i},x\}$ is a circuit or cocircuit of $N$, which is impossible.
Similarly, if $E(T_{i})$ and $E(T_{j})$ fail to be skew and coskew, then $\{e_{i},e_{j}\}$ is a circuit or a cocircuit of $N$.
This contradiction completes the proof.
\end{proof}

\begin{proposition}
\label{degreethreecircuits}
Let $T$ be the canonical decomposition tree for the connected matroid $M$.
Then $T$ has a circuit node with degree at least three if and only if $M$ has a $2$\dash separation $(A,B)$ such that there are wedges $B_{1}$ and $B_{2}$ relative to $A$ satisfying:
\begin{enumerate}[label=\textup{(\roman*)}]
\item $B_{1}\cup B_{2}=B$,
\item $B-B_{1}$ and $B-B_{2}$ each contain at least two elements, and are not coskew, and
\item $A\cap \cl^{*}(B_{1}) = \emptyset$.
\end{enumerate}
\end{proposition}

\begin{proof}
We start by proving the ``only if" direction.
Assume $N$ is a circuit node and that $e_{0},\ldots, e_{n}$ are the edges incident with $N$, where $n\geq 2$.
For $i=0,1,\ldots, n$, let $T_{i}$ be the component of $T\ba e_{i}$ that does not contain $N$.
Let $A=E(T_{0})$, and let $B$ be $E(M) - A$.
Proposition \ref{vertexwedges} tells us that $B_{1}=B-E(T_{1})$ and $B_{2}=B-E(T_{2})$ are wedges relative to $A$.
Then $B_{1}\cup B_{2}=B$.
Moreover, $B-B_{1}=E(T_{1})$ and $B-B_{2}=E(T_{2})$ each contain at least two elements.
Express $M$ as $M_{0}\oplus_{2} M_{1}\oplus_{2} M_{2}$, where $E(M_{i}) = E(T_{i})\cup e_{i}$ for $i=1,2$.
For $i=1,2$, let $C^{*}_{i}$ be a cocircuit of $M_{i}$ that contains $e_{i}$.
As $\{e_{1},e_{2}\}$ is a cocircuit of $N$, and of $M_{0}$, it follows that $(C^{*}_{1}-e_{1})\cup (C^{*}_{2}-e_{2})$ is a cocircuit of $M$.
Thus $B-B_{1}$ and $B-B_{2}$ are not coskew.
Finally, we assume that $a\in A$ is in $\cl^{*}(B_{1})$.
Then $B_{1}\cup a$ is $2$\dash separating, and both $B_{1}\cup a$ and its complement contain at least two elements.
Therefore $B_{1}\cup a$ is displayed by an edge or node of $T$.
This edge or node must be in $T_{0}$, or else it cannot display both $B_{1}$ and $a$.
But now any such displayed set contains all of $B_{2}$ as well as $B_{1}$, and we have a contradiction.

Next we prove the converse.
Let $(A,B)$ be a $2$\dash separation, and let $B_{1}$ and $B_{2}$ be wedges relative to $A$ that satisfy conditions (i), (ii), and (iii).
Assume that $(A,B)$ is not displayed by a circuit or cocircuit node.
Then $(A,B)$ is displayed by an edge $e$ in $T$.
Let $T_{B}$ be the component of $T\ba e$ such that $B=E(T_{B})$, and let $N$ be the node of $T_{B}$ incident with $e$.
Since $(A,B)$ is displayed by $N$, our assumption means $N$ is not a circuit or a cocircuit.
Hence $N$ is a $3$\dash connected component.
Now Proposition \ref{conversecomponents} implies $B_{1}$ and $B_{2}$ are disjoint, so $B-B_{1}=B_{2}$ and $B-B_{2}=B_{1}$ and these wedges are coskew.
This contradicts our assumption, so $(A,B)$ is displayed by a circuit or cocircuit node $N$.
Let $T_{1},\ldots, T_{n}$ be the components of $T-N$ such that $E(T_{i})\subseteq B$ for each $i$, and let $e_{i}$ be the edge joining $N$ to $T_{i}$.

Assume $n\geq 2$, so that we can apply
Proposition \ref{vertexwedges}.
As $B-B_{1}$ and $B-B_{2}$ each have at least two elements, we can assume that $B_{1}=B-E(T_{1})$ and $B_{2}=B-E(T_{2})$.
We express $M$ as $M_{0}\oplus_{2} M_{1}\oplus_{2} M_{2}$, where $E(M_{1}) = E(T_{1})\cup e_{1}$ and $E(M_{2}) = E(T_{1})\cup e_{2}$.
Condition (ii) tells us that there is a cocircuit contained in $(B-B_{1})\cup (B-B_{2}) = E(T_{1})\cup E(T_{2})$ that contains elements from both $E(T_{1})$ and $E(T_{2})$.
This means that $\{e_{1},e_{2}\}$ is a cocircuit in $N$.
This in turn implies that $N$ is a circuit.
If $N$ has degree at least three, then we have nothing left to prove, so we assume this is not the case.
Then $n$ is exactly two, and $A=E(N)-B$.
We can let $C^{*}$ be a cocircuit of $M_{1}$ that contains $e_{1}$.
Then $(C^{*}-e_{1})\cup a$ is a cocircuit of $M$ for any $a\in A$, so $A\cap \cl^{*}(B_{1})\ne \emptyset$, contrary to hypothesis.
Now we must assume that $n < 2$.

If $n=0$, then $B = E(N)-A$.
In this case $B$ has rank or corank equal to one, and a wedge relative to $A$ is $B-b$ for some element $b\in B$.
This is impossible as $B-B_{1}$ has at least two elements.
Therefore $n=1$.

Assume that $B$ contains an element in $E(N)$.
We claim that $B\cap E(N)$ is a wedge relative to $A$.
It is certainly a $2$\dash separating non-empty proper subset of $B$.
If it is properly contained in a $2$\dash separating proper subset of $B$, then that subset is displayed by $e_{1}$, or by a node or an edge in $T_{1}$.
But any such subset contains $A$, and is therefore not a wedge.
Hence $B\cap E(N)$ is a wedge, as we claimed.
Note that $B-b$ is $2$\dash separating for any $b\in B\cap E(N)$.
A simple analysis now shows that the wedges relative to $A$ are $B\cap E(N)$ and $B-b$ for each $b\in B\cap E(N)$.
As $B_{1}\cup B_{2}=B$ implies $B_{1}\ne B_{2}$, we assume without loss of generality that $B_{1}=B-b$ for some $b\in B\cap E(N)$.
This implies $B-B_{1}$ is a singleton set, contrary to hypothesis.
Therefore $B\cap E(N)=\emptyset$.

We let $N_{1}$ be the node joined to $N$ by $e_{1}$, and we observe that $(A,B)$ is displayed by $N_{1}$.
By the previous arguments there is nothing left to prove, unless $N_{1}$ has degree two in $T$, and $B\cap E(N_{1})=\emptyset$.
But this is impossible, as $|E(N_{1})|\geq 3$, and every element of $E(N_{1})$ is either a basepoint, or is contained in $B$.
Now the proof is complete.
\end{proof}

Let $M$ be a connected matroid, and let $(A,B)$ be a $2$\dash separation in $M$.
Assume that the conditions (i) and (ii) from Proposition \ref{characterisingcomponents} are satisfied.
Let $\mathcal{F}_{A}$ be the collection of subsets  $\{A\}\cup\{Z\colon Z\ \text{is a wedge relative to}\ A\}$, and note that the sets in $\mathcal{F}_{A}$ partition $E(M)$, by condition (ii).
We will use the idea that each non-singleton set in $\mathcal{F}_{A}$ is displayed by an edge in $T$, and we will identify the set with the corresponding basepoint element.
Let $\mathcal{I}_{A}$ be the following collection of subsets of $\mathcal{F}_{A}$: we admit $X\subseteq \mathcal{F}_{A}$ to $\mathcal{I}_{A}$ if and only if any circuit of $M$ that is contained in $\cup_{Z\in X}Z$ is contained in some set $Z\in X$.

\begin{proposition}
\label{transducedcomponents}
Let $T$ be the canonical decomposition tree of the connected matroid $M$.
Let $N$ be a $3$\dash connected component incident with the edges $e_{0},e_{1},\ldots, e_{n}$.
For each $i$, let $T_{i}$ be the component of $T\ba e_{i}$ that does not contain $N$.
Set $A$ to be $E(T_{0})$.
Let $\sigma$ be the following bijection from $E(N)$ to $\mathcal{F}_{A}$.
If $x$ is in $E(N) - \{e_{0},\ldots, e_{n}\}$, then set $\sigma(x)$ to be $\{x\}$.
For each $i$ set $\sigma(e_{i})$ to be $E(T_{i})$.
Then $(\mathcal{F}_{A},\mathcal{I}_{A})$ is a matroid, and $\sigma$ is an isomorphism from $N$ to
$(\mathcal{F}_{A},\mathcal{I}_{A})$.
\end{proposition}

\begin{proof}
Note that Proposition \ref{conversecomponents} tells us that the wedges relative to $A$ are $E(T_{1}),\ldots, E(T_{n})$ and the singleton subsets of $E(N)-\{e_{0},\ldots, e_{n}\}$.

We express $M$ as $N\oplus_{2} M_{0}\oplus_{2}\cdots\oplus_{2} M_{n}$, where $E(M_{i}) = E(T_{i})\cup e_{i}$ for each $i$.
Let $X$ be an arbitrary subset of $E(N)$, and let $\{e_{i_{1}},\ldots, e_{i_{t}}\}$ be the intersection of $X$ with $\{e_{0},\ldots, e_{n}\}$.
Then
\[
\sigma(X) = \{E(T_{i_{1}}),\ldots, E(T_{i_{t}})\}\cup \{\{x\}\colon x\in X-\{e_{0},\ldots, e_{n}\}\}.
\]
We claim $X$ is dependent in $N$ if and only if $\sigma(X)\notin \mathcal{I}_{A}$.
Establishing this claim will complete the proof.
Assume for a contradiction that $X$ is minimal with respect to this claim failing.

If $X$ is dependent in $N$, then the minimality of $X$ means it is a circuit of $N$.
For each $i_{k}$ let $C_{i_{k}}$ be a circuit of $N_{i_{k}}$ that contains $e_{i_{k}}$, and note that $C_{i_{k}}-e_{i_{k}}$ is non-empty, since $e_{i_{k}}$ is not a separator of $N_{i_{k}}$.
Now
\[
(X\cup C_{i_{1}}\cup\cdots\cup C_{i_{t}})-\{e_{i_{1}},\ldots, e_{i_{t}}\}
\]
is a circuit of $M$ contained in $\cup_{Z\in\sigma(X)}Z$.
Moreover, this circuit is not contained in a single member of $\sigma(X)$ (as $N$ has no loops).
This tells us that $\sigma(X)$ is not a member of $\mathcal{I}_{A}$, so the claim is satisfied after all.

Now we assume $\sigma(X)$ is not a member of $\mathcal{I}_{A}$.
Let $C$ be a circuit of $M$ contained in $\cup_{Z\in\sigma(X)}Z$ such that $C$ is not contained in any set $Z\in\sigma(X)$.
By the minimality of $X$, we see that $X-\{e_{i_{1}},\ldots, e_{i_{t}}\}\subseteq C$ and $C$ contains elements from each of $E(T_{i_{1}}),\ldots, E(T_{i_{t}})$.
Now it is straightforward to use the definition of $2$\dash sums to verify that $X$ is a circuit of $N$.
\end{proof}

The class of bicircular matroids is not closed under $2$\dash sums, so the next result is not directly useful for our application, but we feel it is likely to be helpful for other projects.
Our proof of Theorem \ref{thm_main} uses the same ideas as Corollary \ref{definedviacomponents}, but requires additional work to analyse the decomposition trees of bicircular matroids.

\begin{corollary}
\label{definedviacomponents}
Let $\mathcal{M}$ be a class of matroids and assume that a matroid is in $\mathcal{M}$ if and only if each connected component is in $\mathcal{M}$, and each connected matroid is in $\mathcal{M}$ if and only if all of its $3$\dash connected components are.
If there is an \mso\ sentence that is satisfied by exactly the $3$\dash connected members of $\mathcal{M}$, then there is a sentence that characterises $\mathcal{M}$.
\end{corollary}

\begin{proof}
We will construct a sentence that is satisfied by a set-system $M=(E,\mathcal{I})$ if and only if $M$ is a connected matroid in $\mathcal{M}$.
The construction is similar to that in Lemma \ref{connectedcase}.
We assume that the sentence $\varphi = Q_{1}X_{i_{1}}\cdots Q_{n}X_{i_{n}}\ \omega$ in prenex normal form characterises the $3$\dash connected matroids in $\mathcal{M}$.

Our sentence is a conjunction with \formula{2-connected} as one term.
Henceforth we assume that $M$ is a connected matroid.
By Proposition \ref{prop_kseparators}, there is an \mso\ formula that is satisfied by $A$ exactly when $(A, E(M)-A)$ is $2$\dash separating in $M$.
Similarly, we have an \mso\dash characterisation of the wedges relative to $A$.
We certainly have \mso\dash characterisations of circuits and cocircuits.
Therefore we can let $\formula{GoodSeparation}[A]$ be a formula that is satisfied if and only if $(A, E(M)-A)$ is a $2$\dash separation that satisfies conditions (i) and (ii) in Proposition \ref{characterisingcomponents}.

Let $\formula{GoodSet}[A,Z]$ be a formula satisfied by $A$, $Z$ if and only if $A$ satisfies $\formula{GoodSeparation}[A]$, and $Z$ can be expressed as a union of wedges relative to $A$ and (possibly) $A$ itself: that is, if and only if $Z$ is a union of members of $\mathcal{F}_{A}$.
Now let $\formula{Independent}[A,Z]$ be a formula satisfied by $A$, $Z$ if and only if $A$, $Z$ satisfies $\formula{GoodSet}[A]$, and any circuit of $M$ contained in $Z$ is contained in either $A$ or a wedge relative to $A$: that is, if $Z$ can be expressed as $\cup_{Z'\in X}$ where $X$ is in $\mathcal{I}_{A}$.

Now we continue to build our sentence by constructing the other term in the conjunction.
We start with the quantification $\forall A\ \formula{GoodSeparation}[A]\ \to$, and then successively replace each occurrence of $\exists X_{i_{k}}$ in $\varphi$ with $\exists X_{i_{k}}\ \formula{GoodSet}[A, X_{i_{k}}]\ \land$ and replacing each $\forall X_{i_{k}}$ with $\forall X_{i_{k}}\ \formula{GoodSet}[A, X_{i_{k}}]\ \to$.
Finally, in $\omega$, we replace each occurrence of $\ind[X_{i_{k}}]$ with $\formula{Independent}[A, X_{i_{k}}]$.

This sentence is satisfied by the connected matroids such that $(\mathcal{F}_{A},\mathcal{I}_{A})$ satisfies $\varphi$, for every $2$\dash separating set $A$ that satisfies conditions (i) and (ii) from Proposition \ref{characterisingcomponents}.
Proposition \ref{transducedcomponents} tells us that $(\mathcal{F}_{A},\mathcal{I}_{A})$ is isomorphic to a $3$\dash connected component.
Furthermore, it follows immediately from Proposition \ref{transducedcomponents} that the quantification $\forall A\ \formula{GoodSeparation}[A]$ applies to every $3$\dash connected component.
So the sentence will be satisfied by those connected matroids such that every $3$\dash connected component is in $\mathcal{M}$.
Hence our constructed sentence characterises the connected matroids in $\mathcal{M}$.
The result now follows from Lemma \ref{connectedcase}.
\end{proof}

\subsection*{Decompositions of bicircular matroids.}
As we have noted before, the class of bicircular matroids is not closed under $2$\dash sums.
This means that characterising connected bicircular matroids is more complicated than characterising the $3$\dash connected members of the class.
We now move towards such a characterisation.

There are two special situations in which we may realise a $2$\dash sum of two bicircular matroids as a $2$\dash sum operation performed on a pair of graphs representing the summands.  
Let $G_1$ and $G_2$ be graphs with $E(G_1) \cap E(G_2) = \{e\}$, where $e$ is not a separator in $B(G_{1})$ and either $e$ is not a separator in $B(G_2)$, or $G_2$ is a cycle. 
When the distinction is important, we refer to an edge that is not a loop as a \emph{link}. 

\begin{enumerate}[label=\textup{(\roman*)}]
\item Suppose that $e$ is a loop in $G_i$ incident with vertex $v_i$, for $i \in \{1, 2\}$.  
The \emph{loop-sum of $G_1$ and $G_2$ on} $e$ is the graph obtained from the disjoint union of $G_1 - e$ and $G_2 - e$ by identifying vertices $v_1$ and $v_2$. 
\item Suppose that $G_2$ is a cycle, and that $e$ is a link in $G_i$ incident with vertices $u_i, v_i$, for $i \in \{1,2\}$.  
The \emph{link-sum} of $G_1$ and $G_2$ \emph{on} $e$ is the graph obtained from the disjoint union of $G_1 - e$ and $G_2 - e$ by identifying $u_1$ with $u_2$ and $v_1$ with $v_2$.  
\end{enumerate}

\begin{proposition}
\label{looplinksums} 
Let $G_1$ and $G_2$ be graphs with $E(G_1) \cap E(G_2) = \{e\}$, where $e$ is not a separator in $B(G_{1})$ and either $e$ is not a separator in $B(G_2)$, or $G_2$ is a cycle.
If $G$ is the loop-sum of $G_1$ and $G_2$ on $e$, then $B(G)$ is the $2$\dash sum of $B(G_1)$ and $B(G_2)$ with $e$ as its basepoint.
If $G_2$ is a cycle and $G$ is the link-sum of $G_1$ and $G_2$ on $e$, then $B(G)$ is the $2$\dash sum along the basepoint $e$ of $B(G_{1})$ and a circuit on the ground set $E(G_2)$. 
\end{proposition}

\begin{proof} 
In either case, it is easily checked that the circuits of $B(G)$ and the $2$\dash sum coincide.
\end{proof}

The next result provides a converse to Proposition \ref{looplinksums}.
Most of the time, a $2$\dash separation in a connected bicircular matroid appears as a loop-sum or a link-sum.

\begin{proposition}
\label{loopseparation}
Let $G$ be a connected graph such that $M=B(G)$ is connected.
Let the ground set of $M$ be $E$.
Let $(A,E-A)$ be a $2$\dash separation such that any series pair of $M$ is contained in either $A$ or $E-A$.
Either $|V(A)\cap V(E-A)|=1$, or one of $G[A]$ or $G[E-A]$ is a path and $V(A)\cap V(E-A)$ is the set of end-vertices of that path.
\end{proposition}

\begin{proof}
Since $M$ is connected, $G$ contains a cycle, so $r(M)$ is $v$, the number of vertices in $G$.
Furthermore, $G$ has no pendent edges.
Let $v_{A}$ and $v_{E-A}$ be the number of vertices in $G[A]$ and $G[E-A]$, and
let $a(A)$ and $a(E-A)$ stand for the number of acyclic components in $G[A]$ and $G[E-A]$, respectively.
Then
\begin{linenomath*}
\begin{multline*}
1 = r(A) + r(E-A) - r(M) = (v_{A}-a(A))+(v_{E-A}-a(E-A))-v\\
=|V(A)\cap V(E-A)| - (a(A)+a(E-A)).
\end{multline*}
\end{linenomath*}
Let $H_{1},\ldots, H_{m}$ be the components of $G[A]$ that contain cycles, and for each $i\in\{1,\ldots, m\}$, let $\alpha_{i}$ be the number of vertices in $H_{i}$ that are in $G[E-A]$, but not leaves of acyclic components in $G[E-A]$.
Now let $T_{1},\ldots, T_{n}$ be the acyclic components of $G[A]$, and let $T_{1}',\ldots, T_{p}'$ be the acyclic components of $G[E-A]$.
For each $i\in\{1,\ldots, n\}$ let $l_{i}$ be the number of leaves in $T_{i}$, and let $\beta_{i}$ be the number of non-leaf vertices in $T_{i}$ that are in $G[E-A]$, but are not leaves of acyclic components in $G[E-A]$.
Finally, for $i\in\{1,\ldots, p\}$, let $l_{i}'$ be the number of leaves in $T_{i}'$.
No vertex can be a leaf in both $T_{i}$ and $T_{j}'$, for this would imply there is a degree-two vertex incident with an edge in $A$ and an edge in $E-A$.
This would imply that there is a series pair not contained in $A$ or $E-A$, in contradiction to the hypotheses.
Therefore each vertex in $V(A)\cap V(E-A)$ is counted by exactly one of the variables $\alpha_{1},\ldots, \alpha_{m}$, $l_{1},\ldots, l_{n}$, $\beta_{1},\ldots, \beta_{n}$, or $l_{1}',\ldots, l_{p}'$.
As each acyclic component has at least two leaves, we conclude that
\begin{linenomath*}
\begin{multline*}
1=\alpha_{1}+\cdots+ \alpha_{m}
+l_{1}+\cdots+ l_{n}
+\beta_{1}+\cdots+\beta_{n}
+l_{1}'+\cdots+ l_{p}' - n - p\\
\geq \alpha_{1}+\cdots+\alpha_{m} +\beta_{1}+\cdots+\beta_{n}
+2(n+p)-(n+p)
\end{multline*}
\end{linenomath*}
If $n>0$, then $n=1$ and $l_{1}=2$.
In this case, $p$, and all the other variables are zero, so $G[A]$ consists of a single acyclic component that has exactly two leaves.
Now it is easy to see that the proposition holds, as it does if $p>0$.
Therefore we assume $n=p=0$.
Hence $G[A]$ and $G[E-A]$ have no acyclic components.
This implies that each $\alpha_{i}$ is the number of vertices in $H_{i}$ that are also in $V(E-A)$.
It follows that $m=1$, and $\alpha_{1}=1$, so the proposition holds once again.
\end{proof}

We next characterise bicircular matroids in terms of their decomposition trees.
Recall that $(M,L)$ is a rooted bicircular matroid if $M=B(G)$ for some graph $G$ in which every element in $L$ is a loop.

\begin{theorem} \label{treebeard}
Let $M$ be a connected matroid and let $T$ be the canonical decomposition tree for $M$. Then $M$ is bicircular if and only if:
\begin{enumerate}[label=\textup{(\roman*)}]
\item every circuit node in $T$ has degree at most two, and 
\item for each $3$\dash connected component, $N$, of $M$, the rooted matroid $(N,L)$ is bicircular, where $L$ is the set of basepoints of $N$ that do not join $N$ to degree-one circuit nodes.
\end{enumerate}
\end{theorem}

\begin{proof}
Assume that $T$ satisfies conditions (i) and (ii).
Note that any circuit with size at least two can be expressed as the bicircular matroid of a graph that comprises two loops joined by a path.
Thus any two elements of the circuit can be represented by loops in a bicircular representation.
(The fact that no more than two elements can be so represented is the reason the number two appears in (i).)
Let $N$ be an arbitrary node in $T$ that is not a degree-one circuit node.
We will set $G_{N}$ to be a graph such that $N = B(G_{N})$.
By conditions (i) and (ii), we can assume that if the basepoint $e$ joins $N$ in $T$ to a node that is not a degree-one circuit node, then $e$ is a loop in $G_{N}$.
Now let $N$ be a degree-one circuit node, and let $N'$ be the node of $T$ adjacent to $N$.
Note that $N'$ is not a circuit node, so $G_{N'}$ has already been defined.
Let $e$ be the edge of $T$ joining $N$ to $N'$.
If $e$ is a link in $G_{N'}$, then we set $G_{N}$ to be a cycle with the edge-set $E(N)$.
If $e$ is a loop in $G_{N'}$, then we define $G_{N}$ to be a cycle on the edge-set $E(N\ba e)$, with the loop $e$ appended to an arbitrary vertex.
Now it is clear from Proposition \ref{looplinksums} that $M$ is a bicircular matroid: in fact, it is the bicircular matroid of the graph we obtain by summing the $G_{N}$ graphs together using loop-sums and links-sums.

Conversely, suppose $M$ is bicircular, 
and let $G$ be a graph with $B(G)=M$.
We can assume $G$ has no isolated vertices, so it is connected.

We proceed by induction on the number of nodes of $T$. 
Assume that $T$ does not satisfy (i) and (ii), and that amongst such counterexamples, we have chosen $M$ so $T$ is as small as possible.

\begin{claim}
\label{deleteleaf}
If $N$ is a degree-one node in $T$, then $T-N$ satisfies conditions (i) and (ii).
\end{claim}

\begin{proof}
Let $e$ be the edge of $T$ incident with $N$.
Then $M$ can be expressed as $N\oplus_{2} M'$, where $e$ displays the separation $(E(N) - e, E(M')-e)$.
Since $M'$ is isomorphic to a minor of $M$, it too is bicircular.
Moreover, $T-N$ is the canonical decomposition tree for $M'$.
As it has fewer vertices that $T$, it follows that $T-N$ satisfies (i) and (ii).
\end{proof}

\begin{claim}
\label{noleaves}
$T$ has no degree-one circuit vertices.
\end{claim}

\begin{proof}
Assume $N$ is a degree-one circuit node.
Then $T-N$ satisfies (i) and (ii), by Claim \ref{deleteleaf}.
Let $N'$ be the node of $T$ adjacent to $N$, and let $e$ be the edge joining $N$ and $N'$.
No edge of $T$ joins a circuit node to a circuit node, so $N'$ is not a circuit.
Furthermore, the basepoint $e$ joins $N'$ to a degree-one circuit node.
Now it is easy to verify that $T$ also satisfies (i) and (ii).
This is a contradiction.
\end{proof}

\begin{claim}
\label{lowdegreecircuits}
$T$ satisfies condition (i).
\end{claim}

\begin{proof}
Assume otherwise, so $T$ has a circuit node, $N$, with more than two neighbours.
If $N'$ is a degree-one node not adjacent to $N$, then $T-N'$ does not satisfy condition (i), contradicting Claim \ref{deleteleaf}.
So $N$ is adjacent to every leaf of $T$, and by the same argument, $N$ has degree exactly three.
Hence $T$ is isomorphic to $K_{1,3}$.
Let the neighbours of $N$ be $M_{1}$, $M_{2}$, and $M_{3}$, and let $e_{1}$, $e_{2}$, and $e_{3}$ be the edges joining these vertices to $N$.

For $i\in \{1,2,3\}$, let $A_{i}$ be $E(M_{i}) - e_{i}$, and let $B_{i} = E(M) - A_{i}$.
It is easy to prove that any series pair in $M$ is contained in a circuit node of $T$.
Therefore the conditions of Proposition \ref{loopseparation} apply to $(A_{i},B_{i})$.
As $M_{i}$ is not a circuit, it is either a cocircuit or is $3$\dash connected.
In either case, $A_{i}=E(M_{i}\ba e_{i})$ is dependent.
So $G[A_{i}]$ is not a path.
The same argument shows that $G[B_{i}]$ is not a path, so there is a unique node, $u_{i}$, in $V(A_{i}) \cap V(B_{i})$.

Let $i$ and $j$ be distinct integers, and assume for a contradiction that $u_{i}=u_{j}$.
We stated earlier that $A_{i}$ is dependent.
Therefore $G$ contains a handcuff containing the node $u_{i}=u_{j}$, as well as a cycle from $G[A_{i}]$ and another cycle from $G[A_{j}]$.
Moreover we can choose this handcuff so that all of its edges are in $A_{i}\cup A_{j}$.
Now we see there is a circuit of $M=N\oplus_{2} M_{1}\oplus_{2} M_{2} \oplus_{2} M_{3}$ contained in $A_{i}\cup A_{j}$.
From the definition of $2$\dash sum, it follows  without difficulty that $\{e_{i},e_{j}\}$ is a circuit in $N$.
This is impossible as $N$ is a circuit with at least three elements.
Hence $u_{1}$, $u_{2}$, and $u_{3}$ are pairwise distinct vertices of $G$.

Assume for a contradiction that $u_{i}$ is in $G[A_{j}]$.
Then $u_{i}$ is incident with an edge in $A_{j}$.
It is also incident with an edge in $A_{i}$, which is necessarily in $B_{j}$.
Hence $u_{i}$ is in $V(A_{j})\cap V(B_{j})$, implying $u_{i}=u_{j}$.
This contradicts the previous paragraph, so $u_{i}$ is not in $G[A_{j}]$.

By changing the labels as necessary, we can assume that $P_{1}$ is a path in $G$ from $u_{1}$ to $u_{3}$ that does not contain $u_{2}$, and $P_{2}$ is a path from $u_{2}$ to $u_{3}$ that does not contain $u_{1}$.
Let $p_{1}$ be an edge that is in $P_{1}$ but not $P_{2}$, and let $p_{2}$ be an edge in $P_{2}$ but not $P_{1}$.
As $u_{3}$ is not in $G[A_{1}]$, it follows that no edge of $P_{1}$ is in $A_{1}$, so $p_{1}\notin A_{1}$.
By the same reasoning we see that $p_{1}$ and $p_{2}$ are not in $A_{1}\cup A_{2}\cup A_{3}$.
Hence $p_{1}$ and $p_{2}$ are in $E(N) - \{e_{1},e_{2},e_{3}\}$.
This means $\{p_{1},p_{2}\}$ is a cocircuit in $M$.
But we can find a loose handcuff of $G$ that contains $P_{1}$, and otherwise contains only edges from $A_{1}$ and $A_{3}$.
Thus there is a circuit of $M$ that intersects $\{p_{1},p_{2}\}$ in a single element, which is impossible.
\end{proof}

As $T$ is a counterexample to the theorem, we deduce from Claim \ref{lowdegreecircuits} that $(N,L)$ is not a bicircular rooted matroid, for some $3$\dash connected component $N$, where $L$ is the set of basepoints in $N$.
(We recall from Claim \ref{noleaves} that no basepoint joins $N$ to a degree-one circuit.)
Let $L$ be $\{e_{1},\ldots, e_{n}\}$, and for each $i$, let $T_{i}$ be the component of $T\ba e_{i}$ not containing $N$.
Let $A_{i}=E(T_{i})$ and let $B_{i} = E(M) - A_{i}$.
We can express $M$ as $N\oplus_{2} M_{1}\oplus_{2}\cdots\oplus_{2} M_{n}$, where $E(M_{i}) = A_{i}\cup e_{i}$.
As in the proof of Claim \ref{lowdegreecircuits}, we can deduce that 
$V(A_{i}) \cap V(B_{i})$ contains a single vertex, $u_{i}$, for each $i$.

\begin{claim}
$u_{i} \in V(E(N)-L)$ for each $i$.
\end{claim}

\begin{proof}
Assume otherwise, so $u_{i}$ is not incident with an edge in $E(N)-L$.
Let $e\in B_{i}$ be an edge of $G$ incident with $u_{i}$.
Since $e$ is not in $E(N)$, it must be in $A_{j}$ for some $j\ne i$.
But $u_{i}$ is incident with an edge in $A_{i}\subseteq B_{j}$, so $u_{i}$ is in $V(A_{j})\cap V(B_{j}) = \{u_{j}\}$.
Both $G[A_{i}]$ and $G[A_{j}]$ contain cycles.
Thus $G[A_{i}\cup A_{j}]$ contains a handcuff using edges from both $A_{i}$ and $A_{j}$ as well as the vertex $u_{i}=u_{j}$.
Hence $A_{i}\cup A_{j}$ contains a circuit of $M$ that uses elements from both $A_{i}$ and $A_{j}$.
Now it follows that $\{e_{i},e_{j}\}$ is a circuit in $N$, which is impossible as $N$ is simple.
\end{proof}

Let $G_{t}$ be the graph obtained from $G$ by adding a loop labelled $e_{t}$ incident with the vertex $u_{t}$, then deleting all edges in $A_{t}$, and deleting any isolated vertices.
The circuits of $B(G_{t})$ are exactly those circuits of $B(G)=M$ that have an empty intersection with $A_{t}$, along with sets of the form $(C-A_{t})\cup e_{t}$, when $C$ is a circuit of $B(G)$ that contains elements from both $A_{t}$ and $B_{t}$.
But these are exactly the circuits of $N\oplus_{2}M_{1}\oplus_{2}\cdots\oplus_{2} M_{t-1}$.
Hence $B(G_{t})=N\oplus_{2}M_{1}\oplus_{2}\cdots\oplus_{2} M_{t-1}$.
Now we similarly construct $G_{t-1}$ from $G_{t}$ by adding the loop $e_{t-1}$ incident with $u_{t-1}$, and then deleting $A_{t-1}$ and isolated vertices.
We see that $B(G_{t-1})=N\oplus_{2}M_{1}\oplus_{2}\cdots\oplus_{2} M_{t-2}$.
Continuing in this way, we reach the conclusion that $N=B(G_{1})$, where $G_{1}$ is obtained from $G$ by adding loops $e_{1},\ldots, e_{t}$, incident with the vertices $u_{1},\ldots, u_{t}$, then deleting $A_{1}\cup\cdots\cup A_{t}$ and any isolated vertices.
But this shows that $(N,L)$ is a rooted bicircular matroid, so we have a contradiction that completes the proof.
\end{proof}

Now we have assembled the tools required to prove our main theorem.

\begin{proof}[Proof of \textup{Theorem \ref{thm_main}}.]
We construct a sentence that is satisfied by a set-system, $M$, if and only if $M$ is a connected bicircular matroid.
Our sentence is a conjunction containing \formula{2-connected} as a term, so we henceforth assume $M$ is a connected matroid.
Let $T$ be the canonical decomposition tree of $M$.
We rely on the characterisation in Theorem \ref{treebeard}.
Thus $M$ is bicircular if and only if $T$ has no circuit node with degree more than two, and the rooted matroid $(N,L)$ is bicircular whenever $N$ is a $3$\dash connected component and $L$ is the set of basepoints that do not join $N$ to degree-one circuit nodes.

By Proposition \ref{prop_kseparators}, we have an \mso\dash characterisation of $2$\dash separations $(A,B)$, and it follows easily that we can characterise when $X$ is a wedge relative to $A$.
We can also characterise when disjoint subsets are coskew.
So it follows by Proposition \ref{degreethreecircuits} that we can characterise when $T$ has a circuit node of degree at least three.
Henceforth we assume $T$ has no such node.

To construct the remainder of the sentence, we rely on Theorem \ref{lem_3concase}.
Thus the formula $\formula{BicircularLoops}[X]$ is satisfied by $(M,X\mapsto L)$ if and only if $(M,L)$ is a $3$\dash connected rooted bicircular matroid.
Assume $\formula{BicircularLoops}[X]$ is expressed in prenex normal form as $Q_{1}X_{i_{1}}\cdots Q_{n}X_{i_{n}}\ \omega[X]$, where $\omega$ is a quantifier-free formula using the variables $\{X_{i_{1}},\ldots, X_{i_{n}},X\}$.

Now we follow the proof of Corollary \ref{definedviacomponents}.
Let $\formula{GoodSeparation}[A]$ be a formula that is satisfied if and only if $(A,E(M)-A)$ obeys conditions (i) and (ii) in Proposition \ref{characterisingcomponents}.
That is, (defining $B$ to be $E(M)-A)$) 
if and only if $\cl(A)\cap B = \emptyset = \cl^{*}(A)\cap B$, and any two distinct wedges relative to $A$ are disjoint, skew, and coskew.
If $\formula{GoodSeparation}[A]$ is satisfied, then $(A,B)$ is displayed by an edge $e$.
We let $T_{B}$ be the component of $T\backslash e$ such that $B=E(T_{B})$.
Let $N$ be the node of $T_{B}$ incident with $e$.
Then $N$ is a $3$\dash connected component of $M$.
Furthermore, if we quantify over all sets $A$ satisfying $\formula{GoodSeparation}[A]$, then Proposition \ref{conversecomponents} guarantees that we have quantified over every $3$\dash connected component of $M$.

Let $e_{1},\ldots, e_{n}$ be the edges of $T_{B}$ incident with $N$ and let $T_{i}$ be the component of $T\ba e_{i}$ not containing $N$, for each $i$.
The wedges relative to $A$ are $E(T_{1}),\ldots, E(T_{n})$, and the singleton subsets of $E(N)-\{e,e_{1},\ldots, e_{n}\}$, by Proposition \ref{conversecomponents}.
It is an easy exercise to prove that $T_{i}$ consists of a single circuit vertex if and only if $E(T_{i})$ is an independent set of $M$.

Let $\formula{LoopWedges}[A,X]$ be a formula that will be satisfied by $A$, $X$ if and only if $A$ satisfies $\formula{GoodSeparation}$, and furthermore every singleton subset of $X$ is contained in either $A$ or a dependent wedge.
Moreover, we insist that every dependent wedge is a subset of $X$.
Furthermore, if $A$ is dependent, then it is also contained in $X$.
Thus $X$ is the union of all dependent wedges, along with $A$ if $A$ is dependent.
So $X$ is $\cup \sigma(e)$, where we use the notation of Proposition \ref{transducedcomponents}, and the union is taken over all basepoints in $N$ that do not join $N$ to a degree-one circuit vertex.
These are exactly the elements that must be represented by loops in a bicircular representation of $N$, according to Theorem \ref{treebeard}.

Now we construct the remainder of our sentence, using the notation established in the proof of Corollary \ref{definedviacomponents}.
We start with the quantification
\[
\forall A\forall X\ \formula{GoodSeparation}[A] \land \formula{LoopWedges}[A,X] \to.
\]

We then modify $\omega$ by successively replacing
each occurrence of $\exists X_{i_{k}}$ with $\exists X_{i_{k}}\ \formula{GoodSet}[A,X_{i_{k}}] \land$ and each occurrence of $\forall X_{i_{k}}$ with $\forall X_{i_{k}}\ \formula{GoodSet}[A,X_{i_{k}}] \to$.
Finally, we replace each occurrence of $\ind[X_{i_{k}}]$ in $\omega[X]$ with the formula $\formula{Independent}[X_{i_{k}}]$.
It follows from Proposition \ref{transducedcomponents} that this sentence will be satisfied if and only if each $3$\dash connected component $N$ satisfies $\formula{BicircularLoops}[X]$, where $X$ is the set of basepoints in $N$ that do not join it to a degree-one circuit vertex.
So our constructed sentence characterises connected bicircular matroid by Theorem \ref{treebeard}.
\end{proof}

\section{Acknowledgements}

We thank the referee for their careful comments.
Funk and Mayhew were supported by a Rutherford Discovery Fellowship, managed by Royal Society Te Ap\={a}rangi.


\end{document}